\documentclass[11pt]{amsart}

\usepackage{verbatim}
\usepackage{eucal,url,amssymb,stmaryrd,
enumerate,amscd,
}

\usepackage{amsfonts}
\usepackage{amsmath,amsthm,amssymb,amscd,enumerate,eucal,url,stmaryrd}

\numberwithin{equation}{section}

\newtheorem{thrm}{Theorem}[section]

\newtheorem{lemma}[thrm]{Lemma}

\newtheorem{prop}[thrm]{Proposition}

\newtheorem{cor}[thrm]{Corollary}

\newtheorem{rmrk}[thrm]{Remark}

\def\frh{{\frak h}}

\begin{document}

\begin{abstract}
We construct explicit compact solutions with non-zero field
strength, non-flat instanton and constant dilaton to the heterotic
string equations in dimensions seven and eight. We present a
quadratic condition on the curvature which is necessary and
sufficient the heterotic supersymmetry and the anomaly cancellation
to imply the heterotic equations of motion in dimensions seven and
eight. We show that some of our examples are compact supersymmetric
solutions of the heterotic equations of motion in dimensions seven
and eight.
\end{abstract}

\title[ Supersymmetric solutions of the heterotic
equations of motion ] {Compact supersymmetric solutions of the
heterotic equations of motion in dimensions 7 and 8}
\date{\today}

\author{Marisa Fern\'andez}
\address[Fern\'andez]{Universidad del Pa\'{\i}s Vasco\\
Facultad de Ciencia y Tecnolog\'{\i}a, Departamento de Matem\'aticas\\
Apartado 644, 48080 Bilbao\\ Spain} \email{marisa.fernandez@ehu.es}

\author{Stefan Ivanov}
\address[Ivanov]{University of Sofia "St. Kl. Ohridski"\\
Faculty of Mathematics and Informatics\\
Blvd. James Bourchier 5\\
1164 Sofia, Bulgaria} \email{ivanovsp@fmi.uni-sofia.bg}

\author{Luis Ugarte}
\address[Ugarte]{Departamento de Matem\'aticas\,-\,I.U.M.A.\\
Universidad de Zaragoza\\
Campus Plaza San Francisco\\
50009 Zaragoza, Spain} \email{ugarte@unizar.es}

\author{Raquel Villacampa}
\address[Villacampa]{Departamento de Matem\'aticas\,-\,I.U.M.A.\\
Universidad de Zaragoza\\
Campus Plaza San Francisco\\
50009 Zaragoza, Spain} \email{raquelvg@unizar.es}

\maketitle

\setcounter{tocdepth}{2} \tableofcontents

\section{Introduction}

The bosonic fields of the ten-dimensional supergravity which arises
as low energy effective theory of the heterotic string are the
spacetime metric $g$, the NS three-form field strength $H$, the
dilaton $\phi$ and the gauge connection $A$ with curvature $F^A$.
The bosonic geometry considered in this paper is of the form
$R^{1,9-d}\times M^d$ where the bosonic fields are non-trivial only
on $M^d$, $d\leq 8$. We consider the two connections
\begin{equation*}
\nabla^{\pm}=\nabla^g \pm \frac12 H,
\end{equation*}
where $\nabla^g$ is the Levi-Civita connection of the Riemannian
metric $g$. Both connections preserve the metric, $\nabla^{\pm}g=0$, and
have totally skew-symmetric torsion $\pm H$, respectively.


The Green-Schwarz anomaly cancellation mechanism requires that the
three-form Bianchi identity receives an $\alpha'$ correction of the
form
\begin{equation}\label{acgen}
dH=\frac{\alpha'}4(p_1(M^p)-p_1(E))= 2\pi^2\alpha'\Big(Tr(R\wedge
R)-Tr(F^A\wedge F^A)\Big),
\end{equation}
where $p_1(M^p), p_1(E)$ are the first Pontrjagin forms of $M^p$
with respect to a connection $\nabla$ with curvature $R$, and that
of the vector bundle $E$ with connection $A$, respectively.

A class of heterotic-string backgrounds for which the Bianchi
identity of the three-form $H$ receives a correction of type
\eqref{acgen} are those with (2,0) world-volume supersymmetry. Such
models were considered in \cite{HuW}. The target-space geometry of
(2,0)-supersymmetric sigma models has been extensively investigated
in \cite{HuW,Str,HP1}. Recently, there is revived interest in these
models \cite{GKMW,CCDLMZ,GMPW,GMW,GPap} as string backgrounds and in
connection to heterotic-string compactifications with fluxes
\cite{Car1,BBDG,BBE,BBDP,y1,y2,y3,y4}.

In writing \eqref{acgen} there is a subtlety to the choice of
connection $\nabla$ on $M^p$ since anomalies can be cancelled
independently of the choice \cite{Hull}. Different connections
correspond to different regularization schemes in the
two-dimensional worldsheet non-linear sigma model. Hence the
background fields given for the particular choice of $\nabla$ must
be related to those for a different choice by a field redefinition
\cite{Sen}. Connections on $M^p$ proposed to investigate the anomaly
cancellation  \eqref{acgen} are $\nabla^g$ \cite{Str,GMW},
$\nabla^+$ \cite{CCDLMZ}, $\nabla^-$ \cite{Berg,Car1,GPap,II}, Chern
connection $\nabla^c$ when $d=6$ \cite{Str,y1,y2,y3,y4}.

A heterotic geometry will preserve supersymmetry if and only if, in
10 dimensions, there exists at least one Majorana-Weyl spinor
$\epsilon$ such that the supersymmetry variations of the fermionic
fields vanish, i.e. the following Killing-spinor equations hold
\cite{Str}
\begin{gather} \nonumber
\delta_{\lambda}=\nabla_m\epsilon = \left(\nabla_m^g
+\frac{1}{4}H_{mnp}\Gamma^{np} \right)\epsilon=\nabla^+\epsilon=0,
\\\label{sup1} \delta_{\Psi}=\left(\Gamma^m\partial_m\phi
-\frac{1}{12}H_{mnp}\Gamma^{mnp} \right)\epsilon=(d\phi-\frac12H)\cdot\epsilon=0, \\
\nonumber
\delta_{\xi}=F^A_{mn}\Gamma^{mn}\epsilon=F^A\cdot\epsilon=0,
\end{gather}
where  $\lambda, \Psi, \xi$ are the gravitino, the dilatino and the
gaugino fields, respectively, and $\cdot$ means Clifford action of
forms on spinors.

The bosonic part of the ten-dimensional supergravity action in the
string frame is \cite{Berg}
\begin{gather}\label{action}
S=\frac{1}{2k^2}\int
d^{10}x\sqrt{-g}e^{-2\phi}\Big[Scal^g+4(\nabla^g\phi)^2-\frac{1}{2}|H|^2
-\frac{\alpha'}4\Big(Tr |F^A|^2)-Tr |R|^2\Big)\Big].
\end{gather}

The string frame field equations (the equations of motion induced
from the action \eqref{action}) of the heterotic string up to
two-loops \cite{HT} in sigma model perturbation theory are (we use
the notations in \cite{GPap})
\begin{gather}\nonumber
Ric^g_{ij}-\frac14H_{imn}H_j^{mn}+2\nabla^g_i\nabla^g_j\phi-\frac{\alpha'}4
\Big[(F^A)_{imab}(F^A)_j^{mab}-R_{imnq}R_j^{mnq}\Big]=0,\\\label{mot}
\nabla^g_i(e^{-2\phi}H^i_{jk})=0,\\\nonumber
\nabla^+_i(e^{-2\phi}(F^A)^i_j)=0.
\end{gather}
The field equation of the dilaton $\phi$ is implied from the first
two equations above.

We search for a solution to lowest nontrivial order in $\alpha'$ of
the equations of motion in dimensions seven and eight that follow
from the bosonic action which also preserves at least one
supersymmetry.

It is known \cite{Bwit,GMPW} (\cite{GPap} for dimension 6) that the
equations of motion of type I supergravity \eqref{mot} with $R=0$
are automatically satisfied if one imposes, in addition to the
preserving supersymmetry equations \eqref{sup1}, the three-form
Bianchi identity \eqref{acgen} taken with respect to a flat
connection $(R=0)$ on $TM$.

A lot of effort had been done in dimension six and compact torsional
solutions for the heterotic/type I string are known to exist
\cite{DRS,BBDG,BBE,CCDLMZ,GMW,y1,y2,y3,y4,DFG,FIUV}.

In dimension five compact supersymmetric solutions to the heterotic
equations of motion with non-zero fluxes and constant dilaton have
been constructed recently in \cite{FIUV2}.

In dimensions 7 and 8 the only known heterotic/type I  solutions
with non-zero fluxes to the equations of motion preserving at least
one supersymmetry (satisfying \eqref{sup1} and \eqref{acgen} without
the curvature term, $R=0$) are those constructed in \cite{FNu,FN,HS}
for dimension 8, and those presented in \cite{GNic} for dimension 7.
All these solutions are noncompact and conformal to a flat space.
Noncompact solutions to \eqref{sup1} and \eqref{acgen} in dimensions
7 and 8 are presented also in \cite{II}.

The main goal of this paper is to construct explicit compact
supersymmetric valid solutions with non-zero field strength,
non-flat instanton and constant dilaton to the heterotic equations
of motion \eqref{mot} in dimensions 7 and 8.

According to no-go (vanishing) theorems  (a consequence of the
equations of motion \cite{FGW,Bwit}; a consequence of the
supersymmetry \cite{IP1,IP2} for SU($n$)-case and \cite{GMW} for the
general case) there are no compact solutions with non-zero flux and
non-constant dilaton satisfying simultaneously the supersymmetry
equations \eqref{sup1} and the three-form  Bianchi identity
\eqref{acgen} if one takes flat connection on $TM$, more precisely a
connection with zero first Pontrjagin 4-form, $Tr(R\wedge R)=0$.

In the compact case one necessarily has to have a non-zero term
$Tr(R\wedge R)$. However, under the presence of a non-zero curvature
4-form $Tr(R\wedge R)$ the solution of the supersymmetry equations
\eqref{sup1} and the anomaly cancellation condition \eqref{acgen}
obeys the second and the third equations of motion (the second and
the third equations in \eqref{mot}) but does not always satisfy the
Einstein equation of motion (the first equation in \eqref{mot}). We
give in Theorem \ref{thac} a quadratic expression for $R$ which is
necessary and sufficient condition in order that \eqref{sup1} and
\eqref{acgen} imply \eqref{mot} in dimensions 7 and 8 based on the
properties of the special geometric structure induced from the first
two equations in \eqref{sup1}. (A similar condition in dimensions
five and six we presented in \cite{FIUV2,FIUV}, respectively.) In
particular, if $R$ is a $G_2$-instanton (resp. $Spin(7)$-instanton)
the supersymmetry equations together with the anomaly cancellation
imply the equations of motion. The latter can also be seen following
the considerations in the Appendix of \cite{GMPW}.

In this article we present compact nilmanifolds in  dimensions seven
and eight satisfying the heterotic supersymmetry equations
\eqref{sup1} with non-zero flux $H$, non-flat instanton  and
constant dilaton obeying the three-form Bianchi identity
\eqref{acgen} with curvature term  $R=R^+$ which is of instanton
type. According to Theorem \ref{thac} these  nilmanifolds are
compact supersymmetric solutions of the heterotic equations of
motion \eqref{mot} in dimensions 7 and 8. The solutions in
dimension~7 are constructed on the $7$-dimensional generalized
Heisenberg nilmanifold, which is a circle bundle over a $6$-torus
with curvature inside the Lie algebra $su(3)$. The $8$-dimensional
compact solutions can be described as a circle bundle over the
product of a $2$-torus by the total space of a circle bundle over a
$4$-torus, or alternatively as the total space of a circle bundle
with curvature inside the Lie algebra $g_2$ over a $7$-manifold
which is a circle bundle over a $6$-torus (see Section
\ref{examples} for details). Based on the examples we present in
Section~\ref{examples} as well as on constructions proposed in
\cite{GMW}, we outline in the last section a more
general construction of compact manifolds solving the first two
equations in \eqref{sup1} with non-constant dilaton depending on reduced number of variables.

Our solutions seem to be the first explicit compact valid
supersymmetric heterotic solutions with non-zero flux, non-flat
instanton and constant dilaton in dimensions 7 and 8 satisfying the
equations of motion \eqref{mot}.

{\bf Our conventions:} We  rise and lower the indices with the
metric and use the summation convention on repeated indices. For
example, $$B_{ijk}C^{ijk}=B_i^{jk}C^i_{jk}=B_{ijk}C_{ijk}\sum_{i,j,k=1}^d B_{ijk}C_{ijk}.$$

The connection 1-forms $\sigma_{ji}$ of a metric connection $\nabla,
\nabla g=0$, with respect to a local basis $\{E_1,\ldots,E_d\}$ are
given by
$$
\sigma_{ji}(E_k) = g(\nabla_{E_k}E_j,E_i),
$$
since we write $\nabla_X E_j = \sigma^s_j(X)\, E_s$.

The curvature 2-forms $\Omega^i_j$ of $\nabla$ are given in terms of
the connection 1-forms $\sigma^i_j$ by
\begin{equation}\label{curvature}
\Omega^i_j = d \sigma^i_j + \sigma^i_k\wedge\sigma^k_j, \quad
\Omega_{ji} = d \sigma_{ji} + \sigma_{ki}\wedge\sigma_{jk}, \quad
R^l_{ijk}=\Omega^l_k(E_i,E_j), \quad R_{ijkl}=R^s_{ijk}g_{ls},
\end{equation}
and the first Pontrjagin class is represented by the 4-form
$$
p_1(\nabla)={1\over 8\pi^2} \sum_{1\leq i<j\leq d}
\Omega^i_j\wedge\Omega^i_j.
$$

\section{General properties of $G_2$ and $Spin(7)$ structures}

We recall the basic properties of the geometric structures induced
from the gravitino and dilatino Killing spinor equations (the first
two equations in \eqref{sup1}) in dimensions 7 and 8.

\medskip
\noindent {\bf $G_2$-structures in $d=7$.} Endow ${\mathbb R}^7$
with its standard orientation and inner product. Let
$\{E_1,\ldots,E_7\}$ be an oriented orthonormal basis and
$\{e^1,\dots,e^7\}$ its dual basis. Consider the three-form $\Theta$
on ${\mathbb R}^7$ given by
\begin{equation}
  \Theta =e^{127} - e^{236} + e^{347}+ e^{567} - e^{146} - e^{245} +
  e^{135}.\label{11}
\end{equation}
The subgroup of $GL(7,\mathbb{R})$ fixing $\Theta$ is the
exceptional Lie group $G_2$.  It is a compact, connected,
simply-connected, simple Lie subgroup of $SO(7)$ of dimension
14~\cite{Br}. The Lie algebra is denoted by $\frak{g}_2$, and it is
isomorphic to the two-forms satisfying 7 linear equations, namely
$\frak{g}_2\cong \Lambda_{14}^2({\mathbb R}^7) =\{\beta\in
\Lambda^2({\mathbb R}^7) \vert *(\beta\wedge\Theta) =- \beta\}$. The
3-form $\Theta$ corresponds to a real spinor $\epsilon$ and
therefore, $G_2$ can be identified as the isotropy group of a
non-trivial real spinor.

The Hodge star operator supplies the 4-form $*\Theta$ given by
\begin{equation}
  *\Theta =  e^{3456} + e^{1457} + e^{1256}+ e^{1234} + e^{2357} +
  e^{1367} - e^{2467}.\label{12}
\end{equation}
The space $\Lambda^2_{14}(\mathbb{R}^7)$ can also be described as
the subspace of 2-forms $\beta$ which annihilate $*\Theta$, i.e.
$\beta\wedge*\Theta=0$. A $7$-dimensional Riemannian manifold $M$ is
called a $G_2$-manifold if its structure group reduces to the
exceptional Lie group $G_2$. The existence of a $G_2$-structure is
equivalent to the existence of a global non-degenerate three-form
which can be locally written as \eqref{11}. The 3-form $\Theta$ is
called \emph{the fundamental form} of the $G_2$-manifold \cite{Bo}.
From the purely topological point of view, a $7$-dimensional
paracompact manifold is a $G_2$-manifold if and only if it is an
oriented spin manifold~\cite{LM}. We will say that the pair
$(M,\Theta)$ is a $G_2$-manifold with \emph{$G_2$-structure}
(determined by) \emph{$\Theta$}.

The fundamental form of a $G_2$-manifold determines a Riemannian
metric \emph{implicitly} through
$g_{ij}=\frac16\sum_{kl}\Theta_{ikl}\Theta_{jkl}$ \cite{Gr}. This is
referred to as the metric induced by $\Theta$.

In~\cite{FG}, Fern\'andez and Gray divide $G_2$-manifolds into 16
classes according to how the covariant derivative of the fundamental
three-form behaves with respect to its decomposition into $G_2$
irreducible components (see also~\cite{CS,GKMW}).  If the
fundamental form is parallel with respect to the Levi-Civita
connection, $\nabla^g\Theta=0$, then the Riemannian holonomy group
is contained in $G_2$. In this case the induced metric on the
$G_2$-manifold is Ricci-flat, a fact first observed by
Bonan~\cite{Bo}. It was shown by Gray~\cite{Gr} (see
also~\cite{FG,Br,Sal}) that a $G_2$-manifold is parallel precisely
when the fundamental form is harmonic, i.e. $d\Theta=d*\Theta=0$.
The first examples of complete parallel $G_2$-manifolds were
constructed by Bryant and Salamon~\cite{BS,Gibb}. Compact examples
of parallel $G_2$-manifolds were obtained first by
Joyce~\cite{J1,J2,J3} and recently by Kovalev~\cite{Kov}.

The Lee form $\theta^7$ is defined by \cite{Cabr}
\begin{equation}\label{g2li}
\theta^7=-\frac{1}{3}*(* d\Theta\wedge\Theta) = \frac{1}{3}*(*
d*\Theta\wedge*\Theta).
\end{equation}
If the Lee form vanishes, $\theta^7=0$, then the $G_2$-structure is
said to be \emph{balanced}. If the Lee form is closed,
$d\theta^7=0$, then the $G_2$-structure is locally conformally
equivalent to a balanced one \cite{FI1}. If the $G_2$-structure
satisfies the condition $d*\Theta=\theta^7\wedge *\Theta$ then it is
called \emph{integrable} and an analog of the Dolbeault cohomology
is investigated in \cite{FUg}. A \emph{cocalibrated} $G_2$-structure
is a balanced $G_2$-structure which is also integrable.

\medskip
\noindent {\bf $Spin(7)$-structures in $d=8$.} Consider ${\mathbb
R}^8$ endowed with an orientation and its standard inner product.
Let $\{E_1,\ldots,E_8\}$ be an oriented orthonormal basis and
$\{e^1,\dots,e^8\}$ its dual basis. Consider the 4-form $\Phi$ on
${\mathbb R}^8$ given by
\begin{eqnarray}\label{s1}
\Phi = \!\!\!& \!\!&  \!\!\! e^{1238} - e^{1347} + e^{1458}+
e^{1678} - e^{1257} - e^{1356} + e^{1246}
 \\ \nonumber &+ \!\!&  \!\!\!
e^{4567} + e^{2568} + e^{2367}+ e^{2345} + e^{3468} + e^{2478} -
e^{3578}.\nonumber
\end{eqnarray}
The 4-form  $\Phi$ is self-dual $*\Phi=\Phi$ and the 8-form
$\Phi\wedge\Phi$ coincides with the volume form of ${\mathbb R}^8$.
The subgroup of $GL(8,\mathbb{R})$ which fixes $\Phi$ is isomorphic
to the double covering $Spin(7)$ of $SO(7)$~\cite{HL}. Moreover,
$Spin(7)$ is a compact simply-connected Lie group of dimension
21~\cite{Br}. The Lie algebra of $Spin(7)$ is denoted by
$\frak{spin}(7)$ and it is isomorphic to the two-forms satisfying 7
linear equations, namely $\frak{spin}(7)\cong \{\beta \in
\Lambda^2(\mathbb{R}^8)|*(\beta\wedge\Phi)=-\beta\}$. The 4-form
$\Phi$ corresponds to a real spinor $\phi$ and therefore, $Spin(7)$
can be identified as the isotropy group of a non-trivial real
spinor.

A \emph{$Spin(7)$-structure} on an 8-manifold $M$ is by definition a
reduction of the structure group of the tangent bundle to $Spin(7)$;
we shall also say that $M$ is a \emph{$Spin(7)$-manifold}. This can
be described geometrically by saying that there exists a nowhere
vanishing global differential 4-form~$\Phi$ on $M$ which can be
locally written as (\ref{s1}). The 4-form $\Phi$ is called the
\emph{fundamental form} of the $Spin(7)$-manifold $M$ \cite{Bo}.

The fundamental form of a $Spin(7)$-manifold determines a Riemannian
metric \emph{implicitly} through
$g_{ij}=\frac{1}{24}\sum_{klm}\Phi_{iklm}\Phi_{jklm}$ \cite{Gr}.
This is referred to as the metric induced by $\Phi$.

In general, not every 8-dimensional Riemannian spin manifold $M$
admits a $Spin(7)$-structure. We explain the precise condition
\cite{LM}. Denote by $p_2(M), {\mathbb X}(M), {\mathbb X}(S_{\pm})$ the
second Pontrjagin class, the Euler characteristic of $M$ and the
Euler characteristic of the positive and the negative spinor
bundles, respectively. It is well known \cite{LM} that a spin
8-manifold admits a $Spin(7)$-structure if and only if ${\mathbb
X}(S_+)=0$ or ${\mathbb X}(S_-)=0$. The latter conditions are
equivalent to $ p_1^2(M)-4\,p_2(M)+ 8\,{\mathbb X}(M)=0$, for an
appropriate choice of the orientation \cite{LM}.

Let us recall that a $Spin(7)$-manifold $(M,g,\Phi)$ is said to be
parallel (torsion-free \cite{J2}) if the holonomy of the metric
$Hol(g)$ is a subgroup of $Spin(7)$. This is equivalent to saying
that the fundamental form $\Phi$ is parallel with respect to the
Levi-Civita connection $\nabla^g$ of the metric $g$. Moreover,
$Hol(g)\subset Spin(7)$ if and only if $d\Phi=0$ \cite{F} (see also
\cite{Br,Sal}) and any parallel $Spin(7)$-manifold is Ricci flat
\cite{Bo}. The first known explicit example of complete parallel
$Spin(7)$-manifold with $Hol(g)=Spin(7)$ was constructed by Bryant
and Salamon \cite{BS,Gibb}. The first compact examples of parallel
$Spin(7)$-manifolds with $Hol(g)=Spin(7)$ were constructed by
Joyce\cite{J1,J2}.

There are 4-classes of $Spin(7)$-manifolds according to the
Fern\'andez classification \cite{F} obtained as irreducible
representations of $Spin(7)$ of the space $\nabla^g\Phi$.

The Lee form $\theta^8$ is defined by \cite{C1}
\begin{equation}\label{c2}
\theta^8 -\frac{1}{7}*(*d\Phi\wedge\Phi)=\frac{1}{7}*(\delta\Phi\wedge \Phi).
\end{equation}
The 4 classes of Fern\'andez classification can be described in
terms of the Lee form as follows \cite{C1}: $W_0 : d\Phi=0; \quad
W_1 : \theta^8 =0; \quad W_2 : d\Phi = \theta^8\wedge\Phi; \quad W :
W=W_1\oplus W_2.$

A $Spin(7)$-structure of the class $W_1$ (i.e. $Spin(7)$-structure
with zero Lee form) is called a \emph{balanced} $Spin(7)$-structure.
If the Lee form is closed, $d\theta^8=0$, then the
$Spin(7)$-structure is locally conformally equivalent to a balanced
one \cite{I1}. It is shown in \cite{C1} that the Lee form of a
$Spin(7)$-structure in the class $W_2$ is closed and therefore such
a manifold is locally conformally equivalent to a parallel
$Spin(7)$-manifold. The compact spaces with closed but not exact Lee
form (i.e. the structure is not globally conformally parallel) have
different topology than the parallel ones \cite{I1}.

Coeffective cohomology and coeffective numbers of Riemannian
manifolds with $Spin(7)$-structure are studied in \cite{Ug}.

\section{The supersymmetry equations}

Geometrically, the vanishing of the gravitino variation is
equivalent to the existence of a non-trivial real spinor parallel
with respect to the metric connection $\nabla^+$ with totally
skew-symmetric torsion $T=H$. The presence of $\nabla^+$-parallel
spinor leads to restriction of the holonomy group $Hol(\nabla^+)$ of
the torsion connection $\nabla^+$. Namely, $Hol(\nabla^+)$ has to be
contained in $SU(3), d=6$
\cite{Str,IP1,IP2,GMW,GIP,CCDLMZ,BBDG,BBE}, the exceptional group
$G_2, d=7$ \cite{FI,GKMW,GMW,FI1}, the Lie group $Spin(7), d=8$
\cite{GKMW,I1,GMW}. A detailed analysis of the induced  geometries
is carried out in \cite{GMW} and all possible geometries (including
non compact stabilizers) are investigated in \cite{GLP,GPRS,GPR,P}.

\subsubsection*{{\bf Dimension $d=7$.}}
The precise conditions to have a solution to the gravitino Killing
spinor equation in dimension 7 were found in \cite{FI}. Namely,
there exists a non-trivial parallel spinor with respect to a
$G_2$-connection with torsion 3-form $T$ if and only if there exists
an integrable $G_2$-structure $(\Theta,g)$, i.e.
$d*\Theta=\theta^7\wedge *\Theta$. In this case, the torsion
connection $\nabla^+$ is unique and the torsion 3-form $T$ is given
by
$$H=T=\frac{1}{6}(d\Theta,*\Theta)\,\Theta - *d\Theta
+*(\theta^7\wedge\Theta).$$
The Riemannian scalar curvature is
\cite{FI1} (\cite{Br1} for the general case)
$s^g=\frac{1}{18}(d\Theta,*\Theta)+||\theta^7||^2 -
\frac{1}{12}||T||^2 + 3\, \delta\theta^7.
$

The necessary conditions to have a solution to the system of
dilatino and gravitino Killing spinor equations were derived in
\cite{GKMW,FI,FI1}, and the sufficiency was proved in \cite{FI,FI1}.
The general existence result \cite{FI,FI1} states that there exists
a non-trivial solution to both dilatino and gravitino Killing spinor
equations in dimension 7 if and only if there exists a
$G_2$-structure $(\Theta,g)$  satisfying the equations
\begin{equation}\label{sol7}
d*\Theta=\theta^7\wedge *\Theta, \quad d\Theta\wedge\Theta=0, \quad
\theta^7=2d\phi.
\end{equation}
Consequently, the torsion 3-form (the flux $H$) is given by
\begin{equation}\label{tsol7}
H=T= -* d\Theta + 2*(d\phi\wedge\Theta).
\end{equation}
The Riemannian scalar curvature satisfies $s^g=8||d\phi||^2
-\frac{1}{12}||T||^2 +6\, \delta d\phi$.

The equations  \eqref{sol7} hold exactly when the $G_2$-structure
$(\bar\Theta=e^{-\frac32\phi}\Theta,\bar g=e^{-\phi}g)$ obeys the
equations
$$d\bar{*}\bar\Theta=d\bar\Theta\wedge\bar\Theta=0,$$ i.e. it is cocalibrated of pure
type.

\subsubsection*{{\bf Dimension $d=8$.}}
It is shown in \cite{I1} that the gravitino Killing spinor equation always has a
solution in dimension 8. Namely, any $Spin(7)$-structure admits a
unique $Spin(7)$-connection with totally skew-symmetric torsion $T$
satisfying
$$T=* d\Phi-\frac76*(\theta^8\wedge\Phi).$$
(In fact, the converse is also true, namely if there are no obstructions to exist
a solution to the gravitino Killing spinor equation then dimension is 8
and the structure is $Spin(7)$ \cite{F3,Nag}.)

The necessary conditions to have a solution to the system of
dilatino and gravitino Killing spinor equations were derived in
\cite{GKMW,I1}, and the sufficiency was proved in \cite{I1}. The
general existence result \cite{I1} states that there exists a
non-trivial solution to both dilatino and gravitino Killing spinor
equations in dimension~8 if and only if there exists a
$Spin(7)$-structure $(\Phi,g)$ with an exact Lee form which is
equivalent to the statement that the $Spin(7)$-structure is
conformally balanced, i.e. the $Spin(7)$ structure
$(\bar\Phi=e^{-\frac{12}7\phi}\Phi,\bar g=e^{-\frac67\phi}g)$
satisfies $\bar{*}d\bar\Phi\wedge\bar\Phi=0$.

The torsion 3-form (the flux~$H$) and the Lee form are given by
\begin{equation}\label{tsol8}
H=T=*d\Phi- 2*(d\phi\wedge\Phi), \qquad \theta^8=\frac{12}{7}d\phi.
\end{equation}
The Riemannian scalar curvature satisfies $s^g=8||d\phi||^2
-\frac{1}{12}||T||^2 + 6\,\delta d\phi$.

In addition to these equations, the vanishing of the gaugino
variation requires the 2-form $F^A$ to be of instanton type
\cite{CDev,Str,HS,RC,DT,GMW}:

\noindent{\bf Case $d=7$:} a $G_2$-instanton, i.e. the gauge field
$A$ is a $G_2$-connection and its curvature 2-form~$F^A\in
\frak{g}_2$. The latter can be expressed in any of the following two
equivalent ways
\begin{equation}\label{7inst}
F^A_{mn}\Theta^{mn}\hspace{0mm}_p=0 \quad \Leftrightarrow \quad
F^A_{mn}= -\frac{1}{2}F^A_{pq}(*\Theta)^{pq}\hspace{0mm}_{mn};
\end{equation}

\noindent{\bf Case $d=8$:} an $Spin(7)$-instanton, i.e. the gauge
field $A$ is a $Spin(7)$-connection and its curvature 2-form $F^A\in
\frak{spin}(7)$. The latter is equivalent to
\begin{equation}\label{8inst}
F^A_{mn}=-\frac{1}{2}F^A_{pq}\Phi^{pq}\hspace{0mm}_{mn}.
\end{equation}

\section{Heterotic supersymmetry and equations of motion}

It is known \cite{Bwit,GMPW} (\cite{GPap} for dimension 6) that the
equations of motion of type I supergravity \eqref{mot} with $R=0$
are automatically satisfied if one imposes, in addition to the
preserving supersymmetry equations \eqref{sup1}, the three-form
Bianchi identity \eqref{acgen} taken with respect to a flat
connection on $TM, R=0$. However, the no-go theorem
\cite{FGW,Bwit,IP1,IP2,GMW} states that if even $Tr(R\wedge R)=0$
there are no compact solutions with non-zero flux $H$ and
non-constant dilaton.

In the presence of a curvature term $Tr(R\wedge R)\not=0$, a
solution of the supersymmetry equations \eqref{sup1} and the anomaly
cancellation condition \eqref{acgen} obeys the second and the third
equations in \eqref{mot} but does not always satisfy the Einstein
equation of motion (the first equation in \eqref{mot}). However if
the curvature $R$ is of instanton type \eqref{sup1} and
\eqref{acgen} imply \eqref{mot}. This can be seen following the
considerations in the Appendix of \cite{GMPW}. We shall give below
an independent proof for the Einstein equation of motion (the first
equation in \eqref{mot}) based on the properties of the special
geometric structure induced from the first two equations in
\eqref{sup1}.

A consequence of the gravitino and dilatino Killing spinor equations is  an
expression of the Ricci tensor $Ric^+_{mn}=R^+_{imnj}g^{ij}$ of the
(+)- connection, and therefore an expression of the Ricci tensor
$Ric^g$ of the Levi-Civita connection, in terms of the suitable
trace of the torsion three-form $T=H$ (the Lee form) and the
exterior derivative of the torsion form $dT=dH$ (see \cite{FI} in
dimension 7 and \cite{I1} in dimension 8). We outline  an
unified proof for dimensions 7 and 8.

Indeed,
the two Ricci tensors are connected by (see e.g. \cite{FI})
\begin{gather}\label{ricg+}
Ric^g_{mn}=Ric^+_{mn}+\frac14T_{mpq}T_n^{pq}-\frac12\nabla^+_sT^s_{mn},
\qquad
Ric^+_{mn}-Ric^+_{nm}=\nabla^+_sT^s_{mn}=\nabla^g_sT^s_{mn},\\\label{mo}
Ric^g_{mn}=\frac12(Ric^+_{mn}+Ric^+_{nm})+\frac14T_{mpq}T_n^{pq}.
\end{gather}

Denote by $\Psi$ the 4-form $-*\Theta$ in dimension 7 or the
$Spin(7)$-form $-\Phi$ in dimension 8. Since $Hol(\nabla^+)\subset
\{
\frak{g}_2, \frak{spin}(7)\}$,  we have the next sequence of
identities
\begin{equation}\label{ric+}
2Ric^+_{mn}=R^+_{mjkl}\Psi_{jkln}=\frac13(R^+_{mjkl}+R^+_{mklj}+
R^+_{mljk})\Psi_{jkln}.
\end{equation}
We apply the following identity established in \cite{FI}
\begin{equation}\label{ric++}
R^+_{jklm}+R^+_{kljm}+R^+_{ljkm}-R^+_{mjkl}-R^+_{mklj}-
R^+_{mljk}\frac32dT_{jklm}-T_{jks}T_{lms} -T_{kls}T_{jms}-T_{ljs}T_{kms}.
\end{equation}
The first Bianchi identity for $\nabla^+$ reads (see e.g.\cite{FI})
\begin{equation}\label{bi+}
R^+_{jklm}+R^+_{kljm}+R^+_{ljkm}=dT_{jklm}-T_{jks}T_{lms}
-T_{kls}T_{jms}-T_{ljs}T_{kms}+\nabla^+_mT_{jkl}.
\end{equation}
Now, \eqref{bi+}, \eqref{ric++} and \eqref{ric+} yield
\begin{equation}\label{ric+f}
Ric^+_{mn}=\frac1{12}dT_{mjkl}\Psi_{jkln}+\frac16\nabla^+_mT_{jkl}\Psi_{jkln}.
\end{equation}
Using the special expression of the torsion \eqref{tsol7} and
\eqref{tsol8} for dimension 7 and 8, respectively,  the equation
\eqref{ric+f} takes the form
\begin{equation}\label{ric+ff}
Ric^+_{mn}=\frac1{12}dT_{mjkl}\Psi_{jkln}-2\nabla^+_md\phi_n\frac1{12}dT_{mjkl}
\Psi_{jkln}-2\nabla^g_md\phi_n+d\phi_sT^s_{mn}.
\end{equation}
Substitute \eqref{ric+ff} into \eqref{mo}, insert the result into
the first equation of \eqref{mot} and use the anomaly cancellation
\eqref{acgen} to conclude
\begin{thrm}\label{thac}
The Einstein equation of motion (the first equation in \eqref{mot})
in dimensions 7 and 8 is a consequence of the heterotic
Killing spinor equations \eqref{sup1} and the anomaly cancellation
\eqref{acgen} if and only if the next identity holds
\begin{equation}\label{supmot}
\frac1{6}\Big[R_{mjab}R_{klab}+R_{mkab}R_{ljab}+R_{mlab}R_{jkab}\Big]\Psi_{jkln}
=R_{mpqr}R_n^{pqr},
\end{equation}
where  the 4-form $\Psi$ is equal  to $-*\Theta$ in dimension 7 and to the
$Spin(7)$-form $-\Phi$ in dimension 8.

In particular, if $R$ is an instanton then  \eqref{supmot} holds.
\end{thrm}
It is shown in \cite{I2} that the curvature of $R^+$ satisfies the
identity $R^+_{ijkl}=R^+_{klij}$ if and only if $\nabla^+_iT_{jkl}$
is a four-form. Now Theorem \ref{thac} yields

\begin{cor}\label{thacp}
Suppose the torsion 3-form is $\nabla^+$-parallel,
$\nabla^+_iT_{jkl}=0$. The equations of motion \eqref{mot} with
respect to  the curvature $R^+$ of the (+)-connection are
consequences of the heterotic Killing spinor equations \eqref{sup1}
and  the anomaly cancellation \eqref{acgen}.
\end{cor}
Manifolds with parallel torsion 3-form are studied in detail in
dimension 6 \cite{Scho} and dimension~7~\cite{Fr7}.

\subsection{Heterotic supersymmetric equations of motion with constant
dilaton}\label{cdil}

In the case when the dilaton is constant we arrive to the following
problems:

\subsubsection*{{\bf Dimension 7}} We look for a compact
$G_2$-manifold $(M,\Theta)$ which satisfies the following conditions
\begin{enumerate}
\item[a).] Gravitino and dilatino Killing spinor  equations with constant dilaton:
Search for a cocalibrated $G_2$-manifold of pure type, i.e.
$d*\Theta=d\Theta\wedge\Theta=0$.
\item[b).] Gaugino equation: look for a  vector bundle
$E$ of rank $r$ over $M$ equiped with a $G_2$-instanton, i.e. a
connection $A$ with curvature 2-form $\Omega^A$ satisfying
\begin{equation}\label{27}
(\Omega^A)_{E_i,E_j}(E_k,E_l)(*\Theta)({E_m,E_n,E_k,E_l})=-2(\Omega^A)_{E_i,E_j}(E_m,E_n),
\end{equation}
where $\{E_1,\ldots,E_7\}$ is a $G_2$-adapted basis on $M$.
\item[c).] Anomaly cancellation condition:
\begin{equation}\label{ac7}
dH=dT=-d*d\Theta= 2\pi^2\alpha' \Big(p_1(M)-p_1(A)\Big), \qquad
\alpha'>0.
\end{equation}
\item[d).] The first Pontrjagin form $p_1(M)$ satisfies equation
\eqref{supmot}.
\end{enumerate}

\subsubsection*{{\bf Dimension 8}} We look for a compact
$Spin(7)$-manifold $(M,\Phi)$ satisfying the following conditions
\begin{enumerate}
\item[a).] Gravitino and dilatino Killing spinor equations with constant dilaton: $(M,\Phi)$
is  balanced, i.e. $*d\Phi\wedge\Phi=0$.
\item[b).] Gaugino equation: look for a  vector bundle
$E$ of rank $r$ over $M$ equiped with an $Spin(7)$-instanton, i.e. a
connection $A$ with curvature 2-form $\Omega^A$ satisfying
\begin{equation}\label{28}
(\Omega^A)_{E_i,E_j}(E_k,E_l)(\Phi)({E_m,E_n,E_k,E_l})=-2(\Omega^A)_{E_i,E_j}(E_m,E_n),
\end{equation}
where $\{E_1,\ldots,E_8\}$ is a $Spin(7)$-adapted basis on $M$.
\item[c).] Anomaly cancellation condition:
\begin{equation}\label{ac8}
dH=dT=d*d\Phi= 2\pi^2 \alpha' \Big(p_1(M)-p_1(A)\Big), \qquad
\alpha'>0.
\end{equation}
\item[d).] The first Pontrjagin form $p_1(M)$ satisfies equation
\eqref{supmot}.
\end{enumerate}

\section{The Lie group setup}\label{examples}

Let us suppose that $g$ is a left invariant Riemannian metric on a
Lie group $G$ of dimension $m$, and let $\{e^1,\ldots,e^m\}$ be an
orthonormal basis of left invariant 1-forms, so that $g=e^1\otimes
e^1 + \cdots +e^m\otimes e^m$. Let
$$
d e^k = \sum_{1\leq i<j \leq m} a_{ij}^k \, e^i\wedge e^j,\quad\quad
k=1,\ldots,m
$$
be the structure equations in the basis $\{e^k\}$.

Let us denote by
$\{E_1,\ldots,E_m\}$ the dual basis. Since $d e^k(E_i,E_j)-e^k([E_i,E_j])$, the Levi-Civita connection 1-forms
$(\sigma^g)^i_j$ are
\begin{equation}\label{lciv}
(\sigma^g)^i_j(E_k) = -\frac12(g(E_i,[E_j,E_k]) - g(E_k,[E_i,E_j]) +
g(E_j,[E_k,E_i]))=\frac12(a^i_{jk}-a^k_{ij}+a^j_{ki}).
\end{equation}

The connection 1-forms $(\sigma^+)^i_j$ for the torsion connection
$\nabla^+$ are given by
\begin{equation}\label{pm}
(\sigma^+)^i_j(E_k)=(\sigma^g)^i_j(E_k) - \frac12 T^i_j(E_k), \qquad
T^i_j(E_k)=T(E_i,E_j,E_k).
\end{equation}

We shall focus on
7 and 8-dimensional nilmanifolds $M=\Gamma\backslash G$ endowed with
an invariant special structure.

\subsection{Explicit solutions in dimension 7}

We consider cocalibrated $G_2$-structures of pure type. From
\eqref{tsol7} we have that the torsion 3-form in this case is given
by
\begin{equation}\label{nabla}
\nabla^+=\nabla^g+\frac12\, T, \quad\quad H=T=-*d\Theta.
\end{equation}
Starting from a balanced SU(3)-structure $(F,\Psi_+,\Psi_-)$ on a
manifold $M^6$ it is easy to see that the $G_2$-structure given by
$\Theta=F\wedge e^7 +\Psi_+$ on the product $M^7=M^6\times S^1$ is
cocalibrated of pure type, where $e^7$ denotes the standard 1-form
on the circle $S^1$. Moreover, following the argument given in
\cite[Theorem 4.6]{II} we conclude that the natural extension of a
SU(3)-instanton on $M^6$ gives rise to a $G_2$-instanton on $M^7$,
and if the torsion connection of the SU(3)-structure satisfies the
modified Bianchi identity then the corresponding $\nabla^+$ given in
\eqref{nabla} also satisfies~\eqref{ac7}. We can apply this to the
compact 6-dimensional explicit solutions given in~\cite{FIUV} to get
compact solutions in dimension~7:

\begin{cor}\label{product-sol-dim7}
Let $(M^6,F,\Psi_+,\Psi_-)$ be a compact balanced SU(3)-nilmanifold
with an SU(3)-instanton solving the modified Bianchi identity for
$\nabla=\nabla^+$ or $\nabla^g$. Then, the $G_2$-manifold
$M^7=M^6\times S^1$ with the structure $\Theta= F\wedge e^7
+\Psi_+$, the $G_2$-instanton obtained as an extension of the
SU(3)-instanton and $\nabla$ being the Levi-Civita connection
$\nabla^g$ or the torsion connection $\nabla^+$ given in
\eqref{nabla}, provides a compact valid solution to the
supersymmetry equations in dimension~7.
\end{cor}

Our goal next is to find more compact $G_2$-solutions to the
supersymmetry equations with non-zero flux and constant dilaton on
non-trivial extensions of the balanced Hermitian structures on the
Lie algebra $\frh_3$ given in \cite{FIUV}. We also provide a new
solution to the equations of motion based on the 7-dimensional
generalized Heisenberg compact nilmanifold.

\bigskip

\noindent{\bf Seven-dimensional extensions of $\frh_3$:} For any
$t\not=0$, the structure equations
\begin{equation}\label{family-h3}
\left\{
  \begin{aligned}
  &d e^1= d e^2= d e^3= d e^4= d e^5 =0, \\
  &d e^6= -2t\, e^{12} + 2t\, e^{34},
  \end{aligned}
\right.
\end{equation}
correspond to the nilpotent Lie algebra $\frh_3=(0,0,0,0,0,12+34)$.
As it is showed in \cite{FIUV}, the SU(3)-structure given by
$$
F=e^{12}+e^{34}+e^{56},\quad\quad \Psi=\Psi_+ + i\,\Psi_-= (e^1+i\,
e^2)(e^3+i\, e^4)(e^5+i\, e^6),
$$
is balanced for all the values of the parameter $t$. Consider any
nilpotent 7-dimensional extension $\frh^7=\frh_3\oplus\langle E_7
\rangle$ such that the $G_2$-structure
\begin{equation}\label{G2-struct}
\Theta=F\wedge e^7 +\Psi_+
\end{equation}
is cocalibrated
of pure type on $\frh^7$, where $e^7$ denotes the dual of $E_7$.
Using \eqref{11} and \eqref{12} it is easy to check that
$d e^7c_0(e^{12} - e^{34}) + c_1(e^{13} + e^{24}) + c_2 (e^{14} -
e^{23})$, where $c_0,c_1,c_2\in \mathbb{R}$. Since $t\not=0$ in
\eqref{family-h3}, we can consider $c_0=0$. Therefore, the nilpotent
Lie algebra $\frh^7$ must be given by the structure equations
\begin{equation}\label{7-example}
\left\{
  \begin{aligned}
  &d e^1= d e^2= d e^3= d e^4= d e^5 =0, \\
  &d e^6= -2t(e^{12} - e^{34}), \\
  &d e^7= c_1(e^{13} + e^{24}) + c_2 (e^{14} - e^{23}), \\
  \end{aligned}
\right.
\end{equation}
where $c_1,c_2\in \mathbb{R}$. Moreover, a direct calculation shows
that the torsion is given by
$$
T=-*d\Theta = -2t(e^{12} - e^{34})e^6 +c_1(e^{13} + e^{24})e^7 +
c_2(e^{14} - e^{23})e^7.
$$
Hence, $dT=-2(4t^2+c_1^2+c_2^2)e^{1234}$. It is easy to prove that
$T$ is parallel with respect to the torsion connection $\nabla^+$ if
and only if $c_1=c_2=0$, which corresponds to the situation
described in Theorem~\ref{product-sol-dim7}.

From (\ref{curvature}), (\ref{lciv}) and (\ref{pm}), it follows that
the non-zero curvature forms $(\Omega^+)^i_j$ of the torsion
connection $\nabla^+$ are
$$
\begin{array}{l}
(\Omega^+)^1_2=- (\Omega^+)^3_4 = -4t^2(e^{12}-e^{34}),\\[8pt]
(\Omega^+)^1_3= (\Omega^+)^2_4 = -c_1^2(e^{13}+e^{24}) - c_1
c_2(e^{14}-e^{23}) + 4 t c_2\, e^{67} ,\\[8pt]
(\Omega^+)^1_4= -(\Omega^+)^2_3 = -c_1 c_2(e^{13}+e^{24}) - c_2^2
(e^{14}-e^{23}) - 4 t c_1\, e^{67},
\end{array}
$$
which implies that the first Pontrjagin form of $\nabla^+$ is
$$
p_1(\nabla^+)=  -\frac{1}{2\pi^2} \left(16 t^4 + (c_1^2 + c_2^2)^2
\right) e^{1234}.
$$

Let us consider $(c_1,c_2)\in \mathbb{Q}^2-\{(0,0)\}$. The
well-known Malcev theorem asserts that the simply-connected
nilpotent Lie group $H^7$ corresponding to the Lie algebra $\frh^7$
has a lattice $\Gamma$ of maximal rank. We denote by $M^7$ the
compact nilmanifold $\Gamma\backslash H^7$.

\begin{lemma}\label{G2-instanton}
Let $A_\lambda$ be the linear connection preserving the metric on
$M^7$ defined by the connection forms
$$
(\sigma^{A_\lambda})^i_j = \lambda\, e^7,\quad\quad
(\sigma^{A_\lambda})^6_7 = e^1 + e^2 + e^3 + e^4 + e^5 + \lambda\,
e^6 + \lambda\, e^7,
$$
for
$(i,j)=(1,2),(1,3),(1,4),(1,5),(2,3),(2,4),(2,5),(3,4),(3,5),(4,5)$,
where $\lambda\in \mathbb{R}$. Then, $A_\lambda$ is a
$G_2$-instanton with respect to the structure \eqref{G2-struct}, and
$$p_1(A_\lambda)= -\frac{\lambda^2}{4 \pi^2} \left( 4t^2+11(c_1^2 + c_2^2)
\right) e^{1234}.$$
\end{lemma}

\begin{proof}
A direct calculation shows that the non-zero curvature forms
$(\Omega^{A_\lambda})^i_j$ of the connection $A_\lambda$ are given
by:
$$
(\Omega^{A_\lambda})^i_j = \lambda c_1(e^{13}+e^{24})+\lambda
c_2(e^{14}-e^{23}),\quad\quad (\Omega^{A_\lambda})^6_7= -2\lambda
t(e^{12}-e^{34}) + \lambda c_1(e^{13}+e^{24})+\lambda
c_2(e^{14}-e^{23}).
$$
for
$(i,j)=(1,2),(1,3),(1,4),(1,5),(2,3),(2,4),(2,5),(3,4),(3,5),(4,5)$.
On the other hand, since the Lie algebra of $G_2$ can be identified
with the subspace of 2-forms which annihilate $*\Theta$ and
$(e^{12}-e^{34})\wedge *\Theta = (e^{13}+e^{24}) \wedge *\Theta
=(e^{14}-e^{23})\wedge *\Theta =0$, the connection $A_\lambda$ is a
$G_2$-instanton.
\end{proof}

As a consequence we get the following compact 7-dimensional
solutions.

\begin{thrm}\label{solution-dim7}
Let $A_\lambda$ be the $G_2$-instanton on $M^7$ given above. If
$\lambda^2 < \min \{ 8t^2, 2(c_1^2+c_2^2)/11 \}$, then
$$
dT = 2\pi^2 \alpha' \, (p_1(\nabla^+)-p_1(A_\lambda)),
$$
with $\alpha'>0$ and $(M^7,\Theta,\nabla^+,A_{\lambda})$ is a
compact solution to the heterotic Killing spinor equations
\eqref{sup1} satisfying the anomaly cancellation condition
\eqref{acgen}.
\end{thrm}

\begin{proof}
Notice that $p_1(\nabla^+)-p_1(A_\lambda)=\frac{1}{4 \pi^2}\left[
4t^2(\lambda^2-8t^2) + (c_1^2+c_2^2) \left(
11\lambda^2-2(c_1^2+c_2^2) \right) \right] e^{1234}$. Therefore, if
$\lambda^2 < \min \{ 8t^2, 2(c_1^2+c_2^2)/11 \}$ then
$p_1(\nabla^+)-p_1(A_\lambda)$ is a negative multiple of $e^{1234}$.
Since $dT=-2(4t^2+c_1^2+c_2^2)e^{1234}$, the result follows.
\end{proof}

\begin{rmrk}
The first Pontrjagin form of the Levi-Civita connection is given by
$$
p_1(\nabla^g)= -\frac{1}{16 \pi^2} \left[ 3(4t^2-c_1^2-c_2^2)^2 +
16t^2(c_1^2+c_2^2) \right] e^{1234},
$$
so there is $\lambda\not=0$ sufficiently small such that $dT =2\pi^2
\alpha' \, (p_1(\nabla^g)-p_1(A_\lambda))$, with $\alpha'>0$.
\end{rmrk}


\bigskip

\noindent{\bf The 7-dimensional generalized Heisenberg group:} Next
we construct a 7-dimensional compact solution to the equations of
motion which is not an extension of the 6-dimensional nilmanifolds
given in~\cite{FIUV}. Let $H(3,1)$ be the 7-dimensional generalized
Heisenberg group, i.e. the nilpotent Lie group consisting of the
matrices of real numbers of the form
$$
H(3,1)=\left\{ \left( \begin{array}{ccccc}
1 & x_1 & x_2 & x_3 & z \\
0 & 1 & 0 & 0 & y_1 \\
0 & 0 & 1 & 0 & y_2 \\
0 & 0 & 0 & 1 & y_3\\
0 & 0 & 0 & 0 & 1
\end{array} \right)
\mid x_i,y_i, z \in \mathbb{R}, 1\leq i\leq 3 \right\}.
$$

We consider the basis for the left invariant 1-forms on $H(3,1)$
given by
$$
e^1= \frac1a d x_1,\ \ e^2= d y_1,\ \ e^3 =
 \frac1b d x_2, \ \ e^4d y_2,\ \ e^5 = \frac1c d x_3, \ \ e^6= d y_3,\ \ e^7 = x_1 dy_1 +
x_2 dy_2 + x_3 dy_3 - dz,
$$
where $a,b,c\in \mathbb{R}-\{0\}$, so that the structure equations
become
\begin{equation}\label{h(3,1)}
\left\{
  \begin{aligned}
  &de^1=de^2=de^3=de^4=de^5=de^6=0, \\
  &de^7= a\, e^{12}+ b\, e^{34} + c\, e^{56}.
  \end{aligned}
\right.
\end{equation}

\begin{lemma}\label{co-pure}
The $G_2$-structure given by
$$
\Theta= (e^{12}+ e^{34} + e^{56})\wedge e^7 + e^{135} - e^{146}
-e^{236} - e^{245}
$$
is cocalibrated for each $a,b,c \in \mathbb{R}-\{0\}$. Moreover,
$\Theta$ is of pure type if and only if $c=-a-b$ or, equivalently,
$de^7 \in \frak{su}(3)$.

\end{lemma}

\begin{proof}
A direct simple calculation shows that $d*\Theta=0$ and
$\Theta\wedge d\Theta= 2 (a+b+c) e^{1234567}$.
\end{proof}

From now on, let us consider $c=-a-b\not=0$ in equations
\eqref{h(3,1)}. The torsion 3-form for the cocalibrated
$G_2$-structure of pure type is given by
$$
T=-*d\Theta=(d e^7)\wedge e^7 = a\, e^{127} + b\, e^{347} -(a+b)
e^{567}.
$$
Hence
\begin{equation}\label{d-torsion}
dT= 2ab\, e^{1234} - 2a(a+b) e^{1256} -2b(a+b) e^{3456}.
\end{equation}
Moreover, it is forward to check that $T$ is parallel with respect
to the torsion connection $\nabla^+$, i.e.

\begin{lemma}\label{parallel-7h}
For any $a,b\in \mathbb{R}-\{0\}$ such that $b\not=-a$, we have
$\nabla^+ T=0$.
\end{lemma}

On the other hand, by (\ref{curvature}), (\ref{lciv}) and (\ref{pm})
we have that the non-zero curvature forms $(\Omega^+)^i_j$ of the
torsion connection $\nabla^+$ are
\begin{equation}\label{7-h-curvature}
\begin{array}{l}
(\Omega^+)^1_2 = -a\left( a\, e^{12} + b\, e^{34} -(a+b) e^{56} \right),\\[8pt]
(\Omega^+)^3_4 = -b\left( a\, e^{12} + b\, e^{34} -(a+b) e^{56} \right),\\[8pt]
(\Omega^+)^5_6= -(\Omega^+)^1_2 - (\Omega^+)^3_4  = (a+b) \left( a\,
e^{12} + b\, e^{34} -(a+b) e^{56} \right).
\end{array}
\end{equation}

Let $\Gamma(3,1)$ denote the subgroup of matrices of $H(3,1)$ with
integer entries and consider the compact nilmanifold
$N(3,1)=\Gamma(3,1)\backslash H(3,1)$. We can describe $N(3,1)$ as a
principal circle bundle over a $6$-torus
 $$
 S^1 \hookrightarrow N(3,1) \to \mathbb{T}^6,
 $$
whose connection 1-form $\eta=e^7$ has curvature $d\eta=a
(e^{12}-e^{56}) + b (e^{34} - e^{56})$ in $\frak{su}(3)$.

Next we show a 3-parametric family of $G_2$-instantons on the
nilmanifold $N(3,1)$.

\begin{prop}\label{G2-inst-h}
Let $A_{\lambda,\mu,\tau}$ be the linear connection on $N(3,1)$
defined by the connection forms
$$
(\sigma^{A_{\lambda,\mu,\tau}})^1_2 = -
(\sigma^{A_{\lambda,\mu,\tau}})^2_1 = \lambda\, e^7,\quad
(\sigma^{A_{\lambda,\mu,\tau}})^3_4 = -
(\sigma^{A_{\lambda,\mu,\tau}})^4_3 = \mu\, e^7,\quad
(\sigma^{A_{\lambda,\mu,\tau}})^5_6 = -
(\sigma^{A_{\lambda,\mu,\tau}})^6_5 = \tau\, e^7,
$$
and $(\sigma^{A_{\lambda,\mu,\tau}})^i_j=0$ for the remaining
$(i,j)$, where $\lambda,\mu,\tau\in \mathbb{R}$. Then,
$A_{\lambda,\mu,\tau}$ is a $G_2$-instanton with respect to the
cocalibrated $G_2$-structure of pure type given in
Lemma~\ref{co-pure} for any $a,b$, $A_{\lambda,\mu,\tau}$ preserves
the metric, and its first Pontrjagin form is given by
$$p_1(A_{\lambda,\mu,\tau})=\frac{\lambda^2+\mu^2+\tau^2}{4\pi^2} \left( ab\, e^{1234} -
a(a+b)\, e^{1256} - b(a+b)\, e^{3456} \right).$$
\end{prop}

\begin{proof}
A direct calculation shows that the non-zero curvature forms
$(\Omega^{A_{\lambda,\mu,\tau}})^i_j$ of the connection
$A_{\lambda,\mu,\tau}$ are:
$$
\begin{array}{l}
(\Omega^{A_{\lambda,\mu,\tau}})^1_2 = \lambda \left( a\, e^{12} +
b\, e^{34} -(a+b) e^{56} \right),\\[8pt]
(\Omega^{A_{\lambda,\mu,\tau}})^3_4 = \mu \left( a\, e^{12} +
b\, e^{34} -(a+b) e^{56} \right),\\[8pt]
(\Omega^{A_{\lambda,\mu,\tau}})^5_6= \tau \left( a\, e^{12} + b\,
e^{34} -(a+b) e^{56} \right).
\end{array}
$$
On the other hand, the Lie algebra of $G_2$ can be identified with
the subspace of 2-forms which annihilate $*\Theta$. Since $( a\,
e^{12} + b\, e^{34} -(a+b) e^{56} )\wedge *\Theta =0$, the
connection $A_{\lambda,\mu,\tau}$ is a $G_2$-instanton.
\end{proof}

The next result gives explicit compact valid solutions on $N(3,1)$
to the heterotic supersymmetry equations with non-zero flux and
constant dilaton satisfying the anomaly cancellation condition which
also solve  the equations of motion \eqref{mot} due to Lemma
\ref{parallel-7h} and Theorem \ref{thac}.

\begin{thrm}\label{N(3,1)}
Let  $N(3,1)$  be the compact cocalibrated of pure type
$G_2$-nilmanifold, $\nabla^+$ be the torsion connection  and
$A_{\lambda,\mu,\tau}$ the $G_2$-instanton given in
Proposition~\ref{G2-inst-h}. If $(\lambda,\mu,\tau)\not=(0,0,0)$ are
small enough so that $\lambda^2+\mu^2+\tau^2 < 2(a^2+ab+b^2)$, then
$$
dT= 2\pi^2 \alpha' \, (p_1(\nabla^+)- p_1(A_{\lambda,\mu,\tau})),
$$
where $\alpha' = 4(2(a^2+ab+b^2)-\lambda^2-\mu^2-\tau^2)^{-1} >0$.

Therefore, the manifold
$(N(3,1),\Theta,\nabla^+,A_{\lambda,\mu,\tau})$ is a compact
solution to the supersymmetry equations \eqref{sup1} obeying the
anomaly cancellation \eqref{acgen} and solving the equations of
motion \eqref{mot}.

The Riemannian metric is locally given by
\begin{equation*}
\begin{array}{rl}
g=  \frac1{a^2}d x_1^2 + d y_1^2 + \frac1{b^2}d x_2^2   + d
y_2^2 + \frac{1}{a^2+b^2} d x_3^2 + d y_3^2 + (x_1 dy_1 + x_2 dy_2
+ x_3 dy_3 - dz)^2.
\end{array}
\end{equation*}

\end{thrm}

\begin{proof}
The non-zero curvature forms of the torsion connection
$\nabla^+_{a,b}$ are given by \eqref{7-h-curvature}, which implies
that its first Pontrjagin form is 
$$
p_1(\nabla^+)=  \frac{a^2 + ab +b^2}{2 \pi^2} \left( ab\, e^{1234}
-a(a+b) e^{1256} -b(a+b) e^{3456} \right).
$$
Now the proof follows directly from \eqref{d-torsion} and
Proposition~\ref{G2-inst-h}. The final assertion in the theorem
follows from Lemma~\ref{parallel-7h} and Theorem~\ref{thac} .
\end{proof}

\begin{rmrk}
The first Pontrjagin form of the Levi-Civita connection is given by
$$
\begin{array}{rl}
p_1(\nabla^g)= \frac{1}{32 \pi^2} [ \!\!\!& \!\! ab(5 a^2+4ab+5
b^2) e^{1234} - a(a+b)(6 a^2+6ab+5 b^2) e^{1256}  \\[8pt]
\!\!& \!\! - b(a+b)(5 a^2+6ab+6 b^2) e^{3456} \, ].
\end{array}
$$
It is easy to see that there is no solution to the heterotic
supersymmetry equations for $\nabla=\nabla^g$ using the instantons
of Lemma~\ref{G2-inst-h}.
\end{rmrk}

\subsection{Explicit solutions in dimension 8}
We consider balanced $Spin(7)$-structures, i.e. $\theta^8=0$. From
\eqref{tsol8} we have that the torsion 3-form in this case is given
by
\begin{equation}\label{nabla-8}
\nabla^+=\nabla^g+\frac12\, T, \quad\quad H=T=*^8d\Phi.
\end{equation}
Starting from a cocalibrated $G_2$-structure of pure type $\Theta$
on a 7-manifold $M^7$ it is easy to see that the $Spin(7)$-structure
given by $\Phi= e^1\wedge\Theta + *^7\Theta$ on the product
$M^8=M^7\times S^1$ is balanced, where $e^1$ denotes the standard
1-form on the circle $S^1$. Moreover, following the argument given
in \cite[Theorem 5.1]{II} we conclude that the natural extension of
a $G_2$-instanton on $M^7$ gives rise to a $Spin(7)$-instanton on
$M^8$, and if the torsion connection of the $G_2$-structure
satisfies the Bianchi identity then the corresponding
$\nabla^+$ given in \eqref{nabla-8} also satisfies~\eqref{ac8}. We
can apply this to the compact 7-dimensional explicit solutions given
in the preceding section to get  compact solutions in
dimension~8:

\begin{cor}\label{product-sol-dim8}
Let $(M^7,\Theta)$ be a compact cocalibrated $G_2$-nilmanifold of
pure type with a $G_2$-instanton solving the modified Bianchi
identity for $\nabla=\nabla^+$ or $\nabla^g$. Then, the
$Spin(7)$-manifold $M^8=M^7\times S^1$ with the structure
$\Phi=e^1\wedge\Theta + *^7\Theta$, the $Spin(7)$-instanton obtained as an
extension of the $G_2$-instanton and $\nabla$ being the Levi-Civita
connection $\nabla^g$ or the torsion connection $\nabla^+$ given in
\eqref{nabla-8}, provides a compact valid solution to the
supersymmetry equations in dimension~8. In particular, starting with
the solutions on the generalized Heisenberg compact nilmanifold
$N(3,1)$ given in Theorem \ref{N(3,1)} one obtains solutions to the
equations of motion in dimension 8 for $\nabla=\nabla^+$.
\end{cor}

Next we find more compact $Spin(7)$-solutions to the supersymmetry
equations with non-zero flux and constant dilaton on non-trivial
extensions of the cocalibrated $G_2$-structures of pure type given
on the 7-dimensional generalized Heisenberg group. Moreover, we also
provide new 8-dimensional solutions to the equations of motion on
some of these non-trivial $Spin(7)$-extensions.

\bigskip

\noindent{\bf Non-trivial Spin(7) extension of the 7-dimensional
generalized Heisenberg group:} Let us consider the 8-dimensional
extension of~\eqref{h(3,1)} given by:
\begin{equation}\label{ext-h(3,1)}
\left\{
  \begin{aligned}
  &de^1 = c\,(e^{24} + e^{25} - e^{34} + e^{35}),\\
  &de^2 = de^3=de^4=de^5=de^6=de^7=0, \\
  &de^8 = a\, e^{23}+ b\, e^{45} - (a+b) e^{67}.
  \end{aligned}
\right.
\end{equation}

These equations correspond to the structure equations of an
8-dimensional nilpotent Lie algebra, which we denote by $\frh^8$.
Let us consider the $Spin(7)$-structure defined by \eqref{s1}.
A direct calculation shows that the torsion is given by
$$T=*d\Phi= c\, e^{124} +c\, e^{125} -c\, e^{134} +c\, e^{135} +
a\, e^{238} + b\, e^{458} - (a+b)\, e^{678}.$$ The torsion satisfies
$T\wedge\Phi=0$ and
\begin{equation}\label{d-torsion-8}
dT = 2(ab-2c^2) e^{2345} -2 a(a+b) e^{2367} -2 b(a+b) e^{4567}.
\end{equation}
There are some special cases for which $T$ is parallel with respect
to the torsion connection, more concretely:

\begin{lemma}\label{par-8}
$\nabla^+ T= 0$ if and only if $(a-b)c=0$.
\end{lemma}

Using again (\ref{curvature}), (\ref{lciv}) and (\ref{pm}), the
non-zero curvature forms $(\Omega^+)^i_j$ of the torsion connection
$\nabla^+$ are given by
$$
\begin{array}{l}
(\Omega^+)^2_3= -a^2\,e^{23}-ab\,e^{45}+a(a+b)\,e^{67},\\[8pt]
(\Omega^+)^2_4= (\Omega^+)^3_5 = (a-b)c\,e^{18} - c^2\,e^{24} -
c^2\,e^{25} +
c^2\,e^{34} -c^2\, e^{35} ,\\[8pt]
(\Omega^+)^2_5= - (\Omega^+)^3_4= -(a-b)c\,e^{18}-c^2\,e^{24} -
c^2\,e^{25} +
c^2\,e^{34} -c^2\, e^{35} ,\\[8pt]
(\Omega^+)^4_5=-ab\,e^{23}-b^2\,e^{45} + b(a+b)\, e^{67} ,\\[8pt]
(\Omega^+)^6_7= a(a+b)\,e^{23}+b(a+b)\,e^{45}-(a+b)^2\,e^{67},
\end{array}
$$
which implies that the first Pontrjagin form $p_1(\nabla^+)$ is
given by
$$
\begin{array}{rl}
2\pi^2\,p_1(\nabla^+) = \!\!& \!\! (ab(a^2+ab+b^2)-4c^4)\,e^{2345} -
a(a+b)(a^2+ab+b^2)\,e^{2367} \\[8pt]
\!\!& \!\! - b(a+b)(a^2+ab+b^2)\,e^{4567}.
\end{array}
$$

Let us denote by $H^8$ the simply-connected nilpotent Lie group
corresponding to the Lie algebra~$\frh^8$. From the explicit
description of the Lie group $H(3,1)$ and from~(\ref{s1}), it
follows that the left invariant metric $g$ on $H^8$ determined by
the $Spin(7)$-structure $\Phi$ can be expressed globally as
\begin{equation}\label{g8}
\begin{array}{rl}
g= \!\!& \!\! \left( dw+ \frac{c}{b}(\frac{x_1}{a}-y_1)dx_2 +
c(\frac{x_1}{a}+y_1)dy_2 \right)^2 + (\frac1a d x_1)^2 + (d
y_1)^2 + (\frac1b d x_2)^2 \\[9pt]
\!\!& \!\!  + (d y_2)^2 + (\frac{-1}{a+b} d x_3)^2 + (d y_3)^2 +
(x_1 dy_1 + x_2 dy_2 + x_3 dy_3 - dz)^2,
\end{array}
\end{equation}
where $(w,x_1,y_1,x_2,y_2,x_3,y_3,z)$ denote the (global)
coordinates of $H^8$, and the $w$-coordinate of the left translation
$L_{(w^0,x_1^0,y_1^0,x_2^0,y_2^0,x_3^0,y_3^0,z^0)}$ by an element
$(w^0,x_1^0,y_1^0,x_2^0,y_2^0,x_3^0,y_3^0,z^0)$ of $H^8$ is given by
$$
w\circ L_{(w^0,x_1^0,y_1^0,x_2^0,y_2^0,x_3^0,y_3^0,z^0)} = w -
\frac{c}{b}(\frac{x_1^0}{a}-y_1^0) x_2 - c(\frac{x_1^0}{a}+y_1^0)
y_2 + w^0.
$$
Notice that the remaining coordinates of
$L_{(w^0,x_1^0,y_1^0,x_2^0,y_2^0,x_3^0,y_3^0,z^0)}$ come easily from
the matrix description of $H(3,1)$.

Let $\Gamma$ be a lattice of maximal rank of $H^8$ and denote by
$M^8$ the compact nilmanifold $\Gamma\backslash H^8$. Clearly, $M^8$
can be described as a circle bundle over the compact $7$-manifold
$N(3,1)$  (defined by \eqref{h(3,1)})
 $$
 S^1 \hookrightarrow M^8 \to N(3,1),
 $$
with connection $1$-form $\eta=e^1$ such that the curvature form
$d\eta = c\,(e^{24} + e^{25} - e^{34} + e^{35}) \in \frak{g}_2$.

Alternatively, the manifold $M^8$ may be viewed as the total space
of a circle bundle over the product of a $2$-torus by a $5$-manifold
$M^5$, which is also the total space of a principal circle bundle
over a $4$-torus, i.e. $S^1 \hookrightarrow M^5 \to \mathbb{T}^4$.
In fact, let $\{e^2,\ldots,e^5\}$ be a basis for the closed
$1$-forms on $\mathbb{T}^4$. Then, $M^5$ is the circle bundle over
$\mathbb{T}^4$ with connection $1$-form $\eta=e^1$ such that the
curvature form is $d\eta = c\,(e^{24} + e^{25} - e^{34} + e^{35})$.
Now, let $e^6$ and $e^7$ be a basis for the closed $1$-forms  on
$\mathbb{T}^2$. Take the product manifold $M^5 \times \mathbb{T}^2$.
Then, $M^8$ is the circle bundle over $M^5 \times \mathbb{T}^2$
 $$
 S^1 \hookrightarrow M^8 \to M^5 \times \mathbb{T}^2,
 $$
with connection form $\nu=e^8$ such that $d\nu = a\,e^{23} +
b\,e^{45} - (a+b)\,e^{67}$.

\begin{prop}\label{Spin7-inst-h}
For each $\lambda,\mu\in \mathbb{R}$, let $A_{\lambda,\mu}$ be the
linear connection on $M^8$ defined by the connection forms:
$$
\begin{array}{rl}
&(\sigma^{A_{\lambda,\mu}})^2_3 = - (\sigma^{A_{\lambda,\mu}})^3_2=
(\sigma^{A_{\lambda,\mu}})^4_5 = -
(\sigma^{A_{\lambda,\mu}})^5_4 = \lambda\, e^8,\\[6pt]
&(\sigma^{A_{\lambda,\mu}})^2_4 = (\sigma^{A_{\lambda,\mu}})^2_5
=(\sigma^{A_{\lambda,\mu}})^3_5
= (\sigma^{A_{\lambda,\mu}})^4_3= - \mu\, e^1, \\[6pt]
&(\sigma^{A_{\lambda,\mu}})^3_4 = (\sigma^{A_{\lambda,\mu}})^4_2 =
(\sigma^{A_{\lambda,\mu}})^5_2
= (\sigma^{A_{\lambda,\mu}})^5_3 =\mu\, e^1, \\[6pt]
&(\sigma^{A_{\lambda,\mu}})^6_7 = -(\sigma^{A_{\lambda,\mu}})^7_6
=-2 \lambda\, e^8,
\end{array}
$$
and $(\sigma^{A_{\lambda,\mu}})^i_j=0$ for the remaining $(i,j)$.
Then, $A_{\lambda,\mu}$ is a $Spin(7)$-instanton with respect to the
$Spin(7)$-structure~\eqref{s1} for any $a,b,c$, $A_{\lambda,\mu}$
preserves the metric, and its first Pontrjagin form is given by
$$2\pi^2\, p_1(A_{\lambda,\mu})=( 3ab\lambda^2 -4c^2\mu^2) e^{2345} - 3 a(a+b)\lambda^2\, e^{2367} -
3 b(a+b)\lambda^2\, e^{4567}.$$
\end{prop}

\begin{proof}
The non-zero curvature forms $(\Omega^{A_{\lambda,\mu}})^i_j$ of the
connection $A_{\lambda,\mu}$ are:
$$
\begin{array}{l}
(\Omega^{A_{\lambda,\mu}})^2_3 =
(\Omega^{A_{\lambda,\mu}})^4_5=\lambda
\left( a\, e^{23}+ b\, e^{45} -(a+b) e^{67} \right),\\[8pt]
(\Omega^{A_{\lambda,\mu}})^2_4 = (\Omega^{A_{\lambda,\mu}})^2_5 = -
(\Omega^{A_{\lambda,\mu}})^3_4 = (\Omega^{A_{\lambda,\mu}})^3_5= -
\mu\, c ( e^{24} + e^{25} - e^{34} + e^{35} ),\\[8pt]
(\Omega^{A_{\lambda,\mu}})^6_7= -(\Omega^{A_{\lambda,\mu}})^2_3 -
(\Omega^{A_{\lambda,\mu}})^4_5  = -2 \lambda \left( a\, e^{23} + b\,
e^{45} -(a+b) e^{67} \right).
\end{array}
$$
Since the Lie algebra of $Spin(7)$ can be identified with the
subspace $\Lambda^2_{21}$ of 2-forms $\beta$ such that
$*(\beta\wedge\Phi)=-\beta$, and since $a\, e^{23} + b\, e^{45}
-(a+b) e^{67}$, $e^{24} + e^{25} - e^{34} + e^{35} \in
\Lambda^2_{21}$ the connection $A_{\lambda,\mu}$ is a
$Spin(7)$-instanton for any $\lambda,\mu$.
\end{proof}

\begin{thrm}\label{M8}
Let $(M^8,\Phi)$ be the compact balanced Spin(7)-nilmanifold,
$\nabla^+$ be the torsion connection  and $A_{\lambda,\mu}$ the
$Spin(7)$-instanton given in Proposition~\ref{G2-inst-h}. If
$(\lambda,\mu)\not=(0,0)$ satisfy $3\lambda^2< a^2+ab+b^2$ and
$3\lambda^2-2\mu^2=a^2+ab+b^2-2c^2$, then
$$
dT= 2\pi^2 \alpha' \, (p_1(\nabla^+)- p_1(A_{\lambda,\mu})),
$$
where $\alpha' = 2(a^2+ab+b^2-3\lambda^2)^{-1} >0$.

Therefore, the manifold $(M^8,\Phi,\nabla^+,A_{\lambda,\mu})$ is a
compact solution to the supersymmetry equations \eqref{sup1}
satisfying the anomaly cancellation \eqref{acgen}.

If $a=b$ then the manifold $(M^8,\Phi,\nabla^+,A_{\lambda,\mu})$
with $(\lambda,\mu)\not=(0,0)$ satisfying $$\lambda^2<a^2,\qquad
3\lambda^2-2\mu^2=3a^2-2c^2$$ is a compact supersymmetric solution
to the heterotic  equations of motion \eqref{mot} in dimension 8.

The Riemannian metric is locally given by \eqref{g8} with $a=b$.
\end{thrm}

\begin{proof}
The proof follows directly from \eqref{d-torsion-8}, the expression
of the first Pontrjagin form of $\nabla^+$ and
Proposition~\ref{Spin7-inst-h}. The final assertion in the theorem
follows from Lemma~\ref{par-8} and Theorem~\ref{thac}.
\end{proof}

\begin{rmrk}
There are also solutions on $M^8$ to the supersymmetry equations
taking $\nabla$ as the Levi-Civita connection $\nabla^g$. For
example, if $a=b=c=1$ in~\eqref{ext-h(3,1)} then a direct
computation shows that the first Pontrjagin form of $\nabla^g$ is
given by
$$
16\pi^2\, p_1(\nabla^g)= -5\, e^{2345} -19\, e^{2367} -19\,
e^{4567}.
$$
From Proposition~\ref{Spin7-inst-h} for $a=b=c=1$ we get
$$
2\pi^2\, p_1(A_{\lambda,\mu})= (3\lambda^2 -4\mu^2) e^{2345} - 6
\lambda^2\, e^{2367} - 6\lambda^2\, e^{4567}.
$$
Since $dT=-2\, e^{2345} -4\, e^{2367} -4\, e^{4567}$, if we choose
the $Spin(7)$-instanton $A_{\lambda,\mu}$ such that $48
\lambda^2<19$ and $64 \mu^2=96 \lambda^2 -9$, then
$$
dT=2\pi^2\, \alpha' \, (p_1(\nabla^g)-p_1(A_{\lambda,\mu})),
$$
where $\alpha'=32(19-48 \lambda^2)^{-1}>0$.
\end{rmrk}

\section{Geometric models}
The structure of the examples
that we have presented  as well as constructions proposed in
\cite{GMW} suggest a more general construction. In this section we
describe  how to derive compact solutions to the system of gravitino
and dilatino Killing spinor equations (the first two equations in
\eqref{sup1}) in dimensions seven and eight starting with a solution
of these equations in low dimensions. The construction is a
$\mathbb T^k$-bundle with curvature of instanton type over a compact
low dimensional solution. The benefit of this construction is the obtained reduction of
the dilaton variables, i.e. the non-constant dilaton depend on reduced number of variables.

First we recall the dimensions 5 and 6.
\begin{itemize}
\item[{\bf D=5}] The gravitino and dilatino Killing spinor
equations in dimension 5  define a reduction of the structure group
$SO(5)$ to $SU(2)$ which is described in terms of differential forms
by Conti and Salamon in \cite{ConS} as follows: an $SU(2)$-structure
on a 5-dimensional manifold $M$ is the quadruplet
$(\eta,\omega_1,\omega_2,\omega_3)$, where $\eta$ is a $1$-form with
a dual vector field $\xi$ and $\omega_i, i=1,2,3$, are $2$-forms on
$M$ satisfying
  \begin{equation*}
  \omega_i\wedge \omega_j=\delta_{ij}v, \quad
  v\wedge\eta\not=0,
  \end{equation*}
for some $4$-form $v$, and $
  X\lrcorner \omega_1=Y\lrcorner \omega_2\Rightarrow \omega_3(X,Y)\ge
  0$, where $\lrcorner$ denotes the interior multiplication.

Let $\mathbb H=Ker\eta$. The 2-forms $\omega_i, i=1,2,3$, can be
chosen to form a basis of the $\mathbb H$-self-dual 2-forms
\cite{ConS}, i.e. $*_{\mathbb H}\omega_i=\omega_i$, where
$*_{\mathbb H}$ denotes the Hodge operator on the 4-dimensional
distribution $\mathbb H$.

Based on analysis done in \cite{FI,FI2} it is shown in \cite{FIUV2}
that

{\it The first two equations in \eqref{sup1} admit a solution  in
dimension five exactly when there exists a five dimensional manifold
$M$ endowed with an $SU(2)$-structure
$(\eta,\omega_1,\omega_2,\omega_3)$ satisfying the structure
equations:
\begin{equation}\label{solstr1}
d\omega_i=2df\wedge \omega_i, \qquad *_{\mathbb H}d\eta = - d\eta
\end{equation}
where $f$ is a smooth function which does not depend on $\xi$,
$df(\xi)=0$.

The flux $H$ is given by $H=T=\eta\wedge d\eta - 2*_4df$ and the
dilaton $\phi$ is equal to $\phi=f + cons.$

Therefore, if the dilaton is constant then the structure equations
are
\begin{equation}\label{solstr}
d\omega_i=0, \qquad *_{\mathbb H}d\eta = - d\eta
\end{equation}
and the flux $H$ is given by $H=T=\eta\wedge d\eta$.

If the $SU(2)$ structure is regular, i.e. the orbit space $N=M/\xi$
is a smooth manifold then $M$ is an $S^1$-bundle over a Calabi-Yau
4-fold (flat torus or K3 surface) with $\mathbb H$-anti-self-dual
curvature form equal to $d\eta$. The metric has the form
$$g=e^{2f}g_{cy}+\eta\otimes\eta,$$
where $g_{cy}$ is the metric on the Calabi-Yau base and $f$ is a
smooth function on the base.}

We do not know whether there exist non-regular $SU(2)$-structures (the integral curves of $\xi$
are not closed) on a compact 5-manifold.

\item[{\bf D=6}] The gravitino and dilatino Killing spinor
equations in dimension 6  define a reduction of the structure group
$SO(6)$ to $SU(3)$ which is described in terms of forms by Chiossi
and Salamon in \cite{CS} as follows:  an $SU(3)$-structure is
$(F,\Psi=\Psi^++\sqrt{-1}\Psi^-)$ with K\"ahler form $F$ and complex
volume form $\Psi$ which satisfy the compatibility relations
$$F\wedge\Psi^{\pm}=0, \qquad \Psi^+\wedge\Psi^-=\frac23F\wedge
F\wedge F.$$

The necessary and sufficient condition for the existence of
solutions to the first two equations in \eqref{sup1} in dimension 6
were derived by Strominger \cite{Str}, namely the manifold should be
complex conformally balanced manifold with non-vanishing holomorphic
volume form $\Psi$ satisfying additional condition. In terms of the
five torsion classes described in \cite{CS}, the Strominger
condition is interpreted in \cite{CCDLMZ} as follows (see \cite{II}
for a slightly different expression):
$$2F\lrcorner dF+\Psi^+\lrcorner d\Psi^+=0.$$
If the dilaton is constant then the Strominger conditions read
\begin{equation}\label{bal}
dF\wedge F=d\Psi^+=d\Psi^-=0.
\end{equation}
Examples of the latter via evolution equations were presented
recently in \cite{FTUV}.

A very promising geometric
model in dimension 6
was proposed in \cite{GP} to be
a certain ${\mathbb T}^2$-bundle over a Calabi-Yau surface (see
\cite{GP} and references therein). Starting with an
$SU(2)$-structure $(\eta,\omega_1,\omega_2,\omega_3)$ on (a regular)
5-manifold $M$ satisfying \eqref{solstr} one  considers an
$S^1$-bundle over M with curvature an exact $\mathbb
H$-anti-self-dual 2 form, $d\alpha$ and the $SU(3)$-structure
$(F,\Psi=\Psi^++\sqrt{-1}\Psi^-)$ defined by
\begin{gather}\label{su3}
F=\omega_1+\eta\wedge\alpha; \qquad \Psi^+=\omega_2\wedge\eta -
\omega_3\wedge\alpha; \qquad \Psi^-=\omega_3\wedge\eta +
\omega_2\wedge\alpha.
\end{gather}
Using \eqref{solstr} and the fact that $d\alpha$ is $\mathbb
H$-anti-self-dual it can be shown following Goldstein and Prokushkin
\cite{GP} that \eqref{bal} hold as a consequence of \eqref{su3}.
When $M$ is regular, i.e. it is an $S^1$-bundle over a Calabi-Yau
4-manifold one gets a holomorphic $\mathbb T^2$-bundle over a
Calabi-Yau surface with anti-self-dual integral curvature 2-forms
which solves the first two equations in \eqref{sup1} with constant
dilaton \cite{GP}. It also follows from considerations in \cite{GP}
that if the starting $SU(2)$-structure solves the equations with
non-constant dilaton, i.e. \eqref{solstr1} hold, then the
$SU(3)$-structure on the  circle bundle also solves the first two
Killing spinor equations with non-constant dilaton in dimension 6.
The ${\mathbb T}^2$-bundle over a K3 surface construction was used
in \cite{y1,y2,y3,y4} to produce the first compact examples in
dimension 6 solving the heterotic supersymmetry equations
\eqref{sup1} with non-zero flux and non-constant dilaton together
with the anomaly cancellation \eqref{acgen} with respect to the
Chern connection.
\end{itemize}
\subsection{$\mathbb T^3$-bundles over a Calabi-Yau surface}
The structure of the example $\Gamma/H^7$, where $H^7$ is the
nilpotent Lie group defined by $\eqref{7-example}$, is generalized
in the following

\begin{thrm}\label{cy3}
Let $\Gamma_i$, $1\leq i \leq 3$, be three closed anti-self-dual
$2$-forms on a Calabi-Yau surface $M^4$, which represent integral
cohomology classes. Denote by $\omega_1$ and by
$\omega_2+\sqrt{-1}\omega_3$ the (closed) K\"ahler form and the
holomorphic volume form on $M^4$, respectively.
Then, there is a compact 7-dimensional manifold $M^{1,1,1}$, which
is
the total space of a  ${\mathbb T}^3$-bundle over $M^4$, and it has
a $G_2$-structure
\begin{equation}\label{g2def}
\Theta=\omega_1\wedge\eta_1+\omega_2\wedge\eta_2-\omega_3\wedge\eta_3+\eta_1\wedge
\eta_2\wedge\eta_3,
\end{equation}
solving the first two Killing spinor equations in \eqref{sup1} with
constant dilaton in dimension $7$, where $\eta_i$, $1\leq i \leq 3$,
is a $1$-form on $M^{1,1,1}$ such that $d\eta_i=\Gamma_i$, $1\leq i
\leq 3$.

For any smooth function  $f$ on $M^4$,  the
$G_2$-structure on $M^{1,1,1}$ given by
\begin{equation}\label{g2deff}
\Theta_f=e^{2f}\Big[\omega_1\wedge\eta_1+\omega_2\wedge\eta_2-
\omega_3\wedge\eta_3\Big]+\eta_1\wedge\eta_2\wedge\eta_3
\end{equation}
solves the first two Killing spinor equations in \eqref{sup1} with
non-constant dilaton $\phi=2f$ (in dimension 7). The metric has the
form
$$g_f=e^{2f}g_{cy}+\eta_1\otimes\eta_1+\eta_2\otimes\eta_2+
\eta_3\otimes\eta_3.$$

\end{thrm}
\begin{proof}
Since  $[\Gamma_i]$, $1\leq i \leq 3$,
define integral cohomology classes on $M^4$, the well-known result
of Kobayashi \cite{Kob} implies that there exists a circle bundle
$S^1 \hookrightarrow M^5 \to M^4$, with connection $1$-form $\eta_1$
on $M^5$ whose curvature form is $d\eta_1 = \Gamma_1$. (From now on,
we write with the same symbol the $2$-form $\Gamma_i$ on $M^4$ and
its lifting to $M^5$ via the projection $M^5 \to M^4$.) Because
$\Gamma_i$ $(i=2,3)$ defines an integral cohomology class on $M^5$,
there exists a principal circle bundle $S^1 \hookrightarrow M^6 \to
M^5$ corresponding to [$\Gamma_2$] and a connection $1$-form
$\eta_2$ on $M^6$ such that $\Gamma_2$ is the  curvature form of
$\eta_2$. Using again the result of Kobayashi, there exists a
principal circle bundle $S^1 \hookrightarrow M^{1,1,1} \to M^6$ with
connection $1$-form $\eta_3$ such that $d\eta_3 = \Gamma_3$ since
$\Gamma_3$ defines an integral cohomology class on $M^6$. The
actions of $S^1$ on each one of the manifolds $M^5$, $M^6$ and
$M^{1,1,1}$ define an action of the $3$-torus on $M^{1,1,1}$ doing
$M^{1,1,1}$ a ${\mathbb T}^3$-bundle over $M^4$.


We have to show that \eqref{g2deff} implies \eqref{sol7}. We
calculate using \eqref{g2deff} that
\begin{gather*}*\Theta_f=e^{2f}\Big[\omega_1\wedge\eta_2\wedge\eta_3+\omega_2\wedge\eta_3\wedge\eta_1-
\omega_3\wedge\eta_1\wedge\eta_2+\frac{e^{2f}}2\omega_1\wedge\omega_1\Big];\\
d\Theta_f=2df\wedge\Theta_f-2df\wedge\eta_1\wedge\eta_2\wedge\eta_3+d\eta_1\wedge\eta_2\wedge\eta_3
-\eta_1\wedge d\eta_2\wedge\eta_3+\eta_1\wedge\eta_2\wedge d\eta_3.
\end{gather*}
From the last two equalities we derive
$$d*\Theta_f=2df\wedge*\Theta_f,  \qquad d\Theta_f\wedge\Theta_f=0,$$
where we have used the equalities  $d\omega_i=0$, $\omega_i\wedge
d\eta_j=0$ $(i,j=1,2,3)$ since $d\eta_j=\Gamma_j$ are anti-self-dual
2-forms on $M^4$, and $df\wedge\omega_i\wedge\omega_i=0$  as a
$5$-form on a four-dimensional Calabi-Yau manifold.
\end{proof}

Notice that in the previous theorem, if we start with a $4$-torus,
we have essentially $3$ possibilities:

\begin{itemize}
\item[1)]
Only one of the three 2-forms $\Gamma_i$ is independent. In this
case, we get $\eqref{7-example}$ with $c_1=c_2=0$. The resulting
compact nilmanifold satisfies the equations of motion.

\item[2)]
Two of the three $2$-forms $\Gamma_i$ are independent. Then, we get
$\eqref{7-example}$ with $(c_1, c_2)\not=(0,0)$. The resulting
compact nilmanifold satisfies the supersymmetric equations but not
the equations of motion.

\item[3)]
 The three  $2$-forms $\Gamma_i$ are independent. In this case, essentially we get the
quaternionic Heisenberg nilmanifold. We did not get any instanton
satisfying the supersymmetric equations, but at least the first $2$
Killing spinor equations are satisfied as the previous theorem
asserts.
 \end{itemize}

\begin{rmrk}
Clearly the conclusions of the above theorem are valid also if we
start with a compact non-regular $M^5$ with an $SU(2)$-structure
satisfying \eqref{solstr1}.  In this case, we take
two anti-self-dual 2-forms
$\Gamma_2$ and $\Gamma_3$ on $M^5$, and we consider $M^{1,1,1}$ the
principal circle bundle over $M^6$ corresponding to $[\Gamma_3]$,
which in turn is a principal circle bundle over $M^5$ corresponding
to $[\Gamma_2]$.
Now, $M^{1,1,1}$ is  a ${\mathbb T}^2$-bundle over $M^5$, and the
$G_2$-structure defined by \eqref{g2def} solves the first two
Killing spinor equations.
\end{rmrk}

Suppose that
$M$ has a $G_2$-structure defined by a $3$-form $\Theta$. Let us
recall that a  $3$-dimensional submanifold $X$ of $M$ is called {\em
associative}, with respect to $\Theta$, if the restriction to $X$ of
$\Theta$ coincides with the Riemannian volume form on $X$ induced by
the $G_2$-metric determined by $\Theta$. (Here we do not assume that
$\Theta$ is closed.) We don't know whether $M^{1,1,1}$ has a
$G_2$-structure, defined by a $3$-form $\Theta$, such the fibers are
associative with respect to $\Theta$.

In  \cite{GP}, it is proved that certain non-trivial ${\mathbb
T}^2$-bundles $M$ over a Calabi-Yau surface have a natural complex
structure not admitting K\"ahler metric. The key of his proof is
that the fibers are complex submanifolds of $M$.  For the previous
construction of ${\mathbb T}^3$-bundles $M^{1,1,1}$ over a
Calabi-Yau surface we have

\begin{lemma}\label{associative}
In the conditions of Theorem \ref{cy3}, suppose that one of the
integral cohomology classes represented by $\Gamma_i$ is non-trivial
on $M^4$. Let $\Theta$ be a $3$-form defining a $G_2$-structure on
$M^{1,1,1}$, such that there is a fibre ${\mathbb T}^3$ which is
associative with respect to $\Theta$. Then $\Theta$ is not closed.
Therefore, the $G_2$-structure on $M^{1,1,1}$ is non-parallel.
\end{lemma}

\begin{proof}
We know that one of the circle bundles considered in the
construction of $M^{1,1,1}$ is non-trivial since one of the forms
$\Gamma_i$ defines a non-zero cohomology class on $M^4$. Then, one
can check that the homology class in $H_3(M^{1,1,1},\mathbb R)$
defined by the fibres is trivial. Therefore, if some ${\mathbb T}^3$
fibre is associative, then $\Theta$ cannot be closed. Otherwise,
there is a well-defined cohomology class [$\Theta$] in
$H^3(M^{1,1,1},\mathbb R)$ and it evaluates on [${\mathbb T}^3$] to
give a positive number, i.e. the volume of ${\mathbb T}^3$, which is
a contradiction with the triviality of [${\mathbb T}^3$] .
\end{proof}


\subsection{$\mathbb S^1$-bundles over a manifold with a balanced $SU(3)$-structure}
Next result generalizes the structure of the example $N(3,1)$
defined by $\eqref{h(3,1)}$.

\begin{thrm}\label{thsug2}
Let $M^6$ be a compact complex 6-manifold solving the first two
Killing spinor equations with constant dilaton in dimension 6, i.e.
there exists an $SU(3)$-structure $(F,\Psi^+,\Psi^-)$ satisfying
\eqref{bal}. Let $\Gamma$ be a closed integral 2-form which is an
$SU(3)$-instanton, $\Gamma\in su(3)$, i.e.
$\Gamma_{\alpha\beta}=\Gamma_{\bar\alpha\bar\beta}=\Gamma_{\alpha\bar\beta}F^{\alpha\bar\beta}=0$
in local holomorphic coordinates. Then, there is a principal circle
bundle
$\pi:M^7\longrightarrow M^6$ with a connection form $\eta$ such that
$\Gamma=d\eta$ is the curvature of $\eta$ and
the $G_2$-structure
\begin{equation}\label{g2su3}
\Theta=F\wedge\eta+\Psi^+,  \qquad *\Theta=\frac12F\wedge
F+\Psi^-\wedge\eta
\end{equation}
solves the first two Killing spinor equations in \eqref{sup1} with
constant dilaton.

\end{thrm}
\begin{proof}
The exterior derivative of \eqref{g2su3}, with the help of
\eqref{bal}, yields
$$d*\Theta=\frac12d(F\wedge F) +d\Psi^-\wedge\eta - \Psi^-\wedge
d\eta=0,$$ and
$$d\Theta\wedge \Theta = F^2\wedge d\eta \wedge \eta +
(F\wedge\eta+\Psi^+)\wedge d\Psi^+ - dF\wedge\Psi^+\wedge\eta=0,$$
because of the algebraic facts $\Psi^-\wedge d\eta=0$, $F^2\wedge
d\eta=0$ since $d\eta\in su(3)$, and because $dF\wedge\Psi^+=0$ on a
complex manifold (see e.g. \cite{CS}). Hence, \eqref{sol7} hold with
$\theta^7=0$.

The existence of a principle circle bundle in the conditions above
follows again from~\cite{Kob}.
\end{proof}



\subsection{$\mathbb S^1$-bundles over a cocalibrated
$G_2$-manifold of pure type} We describe a more general situation
inspired by the structure of the example $\Gamma/H^8$ defined by
$\eqref{ext-h(3,1)}$ and by considerations in \cite{GMW}.

\begin{thrm}\label{thsp7g2}
Let $M^7$ be a compact $G_2$-manifold solving the first two
equations of \eqref{sup1} with constant dilaton in dimension 7, i.e.
there exists a $G_2$-structure $\Theta$ satisfying
$d*\Theta=d\Theta\wedge\Theta=0$. Let $f$ be a smooth function on
$M^7$, and  let $\Gamma_4$ be a closed integral 2-form on $M^7$
which is a $G_2$-instanton, $\Gamma_4\in g_2$, i.e. it satisfies
\eqref{7inst}. Then, we have
\begin{itemize}
\item[i)]
There is a principal circle bundle $\pi:M^8\longrightarrow M^7$
corresponding to $[\Gamma_4]$ and a connection $1$-form $\eta_4$ on
$M^8$ whose curvature form is $\Gamma_4$, such that the
$Spin(7)$-structure
\begin{equation}\label{spin7g2f}
\Phi_f=e^{3f}\Theta\wedge\eta_4+e^{4f}*_7\Theta,
\end{equation}
solves the first two Killing spinor equations in \eqref{sup1} with
non-constant dilaton $\phi=2f$ in dimension 8, where $*_7$ denotes
the Hodge star operator on $M^7$.
The $Spin(7)$-metric has the form
$$g_f=e^{2f}g_7+\eta_4\otimes\eta_4.$$
\item[ii)]
If $M^7$ is a circle bundle over a compact $6$-manifold
$(M^6,F,\Psi^+,\Psi^-)$ as in
Theorem~\ref{thsug2}, $f$ is a smooth function on $M^6$ and  the
form $\Gamma_4$ of the part i) is such that $\Gamma_4\in su(3)$,
 then  there is a compact
8-dimensional manifold $M^{1,1}$ with a free structure preserving
$\mathbb T^2$-action and a fibration $\pi:M^{1,1}/\mathbb T^2\cong
M^6$ with the $Spin(7)$-structure
\begin{equation}\label{spinsu}
\Phi_f=e^{3f}\Big[F\wedge\eta+\Psi^+\Big]\wedge\eta_4+e^{4f}\Big[\frac12F\wedge
F+\Psi_-\wedge\eta\Big],
\end{equation}
solving the first two Killing spinor equations in \eqref{sup1} with
non-constant dilaton $\phi=2f$ in dimension 8, where $\eta$ is the
connection $1$-form on the circle bundle over $M^6$ corresponding to
$\Gamma$. The metric has the form
$$g_f=e^{2f}(g_6+\eta\otimes\eta)+\eta_4\otimes\eta_4.$$
\end{itemize}
\end{thrm}
\begin{proof}
 To prove i) first we show that the Lee form
$7\theta^8_f=-*(*d\Phi\wedge\Phi)$ is an exact 1-form. The exterior
derivative of \eqref{spin7g2f} yields
$$d\Phi_f=3e^{3f}df\wedge\Theta\wedge\eta_4 +e^{3f}d\Theta\wedge\eta_4+4e^{4f}df\wedge*_7\Theta
-e^{3f}\Theta\wedge d\eta_4.
$$
 The latter leads to
$$*d\Phi_f=-3e^{4f}*_7(df\wedge\Theta) +e^{4f}*_7d\Theta+4e^{3f}*_7(df\wedge*_7\Theta)\wedge\eta_4 +
2e^{4f}d\eta_4\wedge\eta_4,$$
where we have used the well known fact
that $*_7(\Theta\wedge d\eta_4)=-2d\eta_4$ since $d\eta\in g_2$.

Consequently, we claim
\begin{multline}\label{bas78}
*d\Phi_f\wedge\Phi_f=-3e^{7f}*_7(df\wedge\Theta)\wedge\Theta\wedge\eta_4+
4e^{7f}*_7(df\wedge*_7\Theta)\wedge*_7\Theta\wedge\eta_4+\\
3e^{8f}*_7(df\wedge\Theta)\wedge*_7\Theta+e^{7f}*_7d\Theta\wedge\Theta\wedge\eta_4
+e^{8f}*_7d\Theta\wedge*_7\Theta+2e^{8f}*_7\Theta\wedge d\eta_4\wedge\eta_4\\=
24e^{7f}*_7df\wedge\eta_4.
\end{multline}
Indeed, the second line in \eqref{bas78} gives not contribution
since the first term  vanishes because it is a general algebraic
identity valid on any $G_2$-manifold, the second term  is zero due
to the second equality in \eqref{g2li}, the third term is zero
because of the following chain of equalities
$$*_7d\Theta\wedge*_7\Theta=g(*_7d\Theta,\Theta)vol._7=g(d\Theta,*_7\Theta)vol._7=d\Theta\wedge\Theta=0$$
and the fourth term is zero because $*_7\Theta\wedge d\eta_4=0$ since $d\eta_4\in g_2$.

The terms in the first line are subject to the following well known
algebraic $G_2$-identities
$$*_7(df\wedge\Theta)\wedge\Theta=-4*_7df,\qquad
*_7(df\wedge*_7\Theta)\wedge*_7\Theta=3*_7df.$$

Hence, we obtain from \eqref{c2} and \eqref{bas78} that
$\theta^8_f=\frac{24}7df$, i.e. the Lee form is an exact form which
completes the proof of
i). The existence of the principle circle bundle $S^1
\hookrightarrow M^8 \to M^7$ in the conditions above follows from
the result of Kobayashi \cite{Kob}.

Now, let us suppose that $\Gamma$ and $\Gamma_4$ are closed integral
$2$-forms on $M^6$, such that $\Gamma$ and $\Gamma_4\in su(3)$. Let
$M^7$ be the principal circle bundle over $M^6$ corresponding to
[$\Gamma$] as in Theorem~\ref{thsug2}. Since [$\Gamma_4$] defines an
integral cohomology class on $M^7$, Kobayashi theorem implies that
there exists a principal circle bundle $S^1 \hookrightarrow M^{1,1}
\to M^7$ corresponding to [$\Gamma_4$] and a connection $1$-form
$\eta_4$ whose curvature is $\Gamma_4$. The actions of $S^1$ on each
one of the manifolds $M^7$ and $M^{1,1}$ define an action of the
$2$-torus on $M^{1,1}$ and $M^{1,1}$ can be considered a ${\mathbb
T}^2$-bundle over $M^6$. Substituting (\ref{g2su3}) in
(\ref{spin7g2f}), and using Theorem~\ref{thsug2} and the part i), we
conclude ii).

\end{proof}


\begin{rmrk}
In Theorem~\ref{thsp7g2}, if $M^7$ is a $\mathbb T^2$-bundle over a
compact non-regular $M^5$ as in Remark $6.2$, such that $M^5$ has an
$SU(2)$-structure $(\eta_1,\omega_1,\omega_2,\omega_3)$ satisfying
\eqref{solstr1}, and there exist three closed anti-self-dual 2-forms
$\Gamma_2$, $\Gamma_3$ and $\Gamma_4$ on $M^5$ representing integral
cohomology classes, then the $S^1$-bundle over $M^7$, constructed in
Theorem~\ref{thsp7g2}, is a $\mathbb T^3$-bundle over $M^5$ with
$Spin(7)$-structure
$$
\Phi_f=e^{3f}\Theta_{f}\wedge\eta_4+*_7\Theta_{f},
$$
solving the first two equations in \eqref{sup1} with non-constant
dilaton, where the $G_2$-form $\Theta_{f}$ on $M^7$ is given by
\eqref{g2deff}. The $Spin(7)$-metric is
$$g_f=e^{2f}(g_5+\eta_2\otimes\eta_2+
\eta_3\otimes\eta_3)+\eta_4\otimes\eta_4,$$ where $f$ and $g_5$
denote a smooth function and the metric on $M^5$, respectively.

Moreover, we must notice that in Theorem~\ref{thsp7g2}, if $M^7$ is
a $\mathbb T^3$-bundle over a Calabi-Yau surface as in
Theorem~\ref{cy3}, and the form $\Gamma_4$ considered in Theorem
\ref{thsp7g2} is such that $\Gamma_4\in su(2)$, i.e. anti-self-dual
2-form on $M^4$, then the $S^1$-bundle constructed in
Theorem~\ref{thsp7g2} is
a $\mathbb T^4$-bundle over the Calabi-Yau $M^4$ with a
$Spin(7)$-structure given by
$$
\Phi=\Theta_{f}\wedge\eta_4+*_7\Theta_{f},
$$
which solves the first two equations in \eqref{sup1} with
non-constan dilaton, where the $G_2$-form $\Theta_{f}$ is given by
\eqref{g2deff}. The metric is given by
$$g_f=e^{2f}g_{cy}+\eta_1\otimes\eta_1+\eta_2\otimes\eta_2+
\eta_3\otimes\eta_3+\eta_4\otimes\eta_4.$$

Suppose that one of the integral cohomology classes represented by
$\Gamma_i$ is non-trivial on $M^4$. Let $\Phi$ be a $4$-form
defining a $Spin(7)$-structure on the total space of the
$S^1$-bundle over $M^7$, such that there is a fibre ${\mathbb T}^4$
which is associative with respect to $\Phi$. Then we conclude that
$\Phi$ is not closed similarly as in the proof of Lemma
\ref{associative}. Therefore, the $Spin(7)$-structure on the total
space of the $S^1$-bundle over $M^7$ is non-parallel.
\end{rmrk}

\medskip
\noindent {\bf Acknowledgments.}  We would like to thank Alexei
Kovalev for useful discussions about associative submanifolds. This
work has been partially supported through grant MEC (Spain) MTM
2005-08757-C04-02 and under project "Ingenio Mathematica (i-MATH)"
No. CSD2006-00032 (Consolider Ð Ingenio 2010). The second author is
partially supported by the Contract 154/2008 with the University of
Sofia `St.Kl.Ohridski`. S.I. is a Senior Associate to the Abdus
Salam ICTP, Trieste and the final stage of the research was done
during his stay in the ICTP, Fall 2008.


\begin{thebibliography}{33}


\bibitem {BBDG} K. Becker, M. Becker, K. Dasgupta, P.S. Green,
{\em Compactifications of Heterotic Theory on Non-K\"ahler Complex
Manifolds: I}, JHEP 0304 (2003) 007.

\bibitem{BBE}  K. Becker, M. Becker, K. Dasgupta, P.S. Green, E. Sharpe,
{\em Compactifications of Heterotic Strings on Non-K\"ahler Complex
Manifolds: II}, Nucl. Phys. {\bf B678} (2004), 19-100.

\bibitem{BBDP} K. Becker, M. Becker, K. Dasgupta, S. Prokushkin,
{\em Properties from heterotic vacua from superpotentials},
hep-th/0304001.

\bibitem{y4} K. Becker, M. Becker, J-X. Fu, L-S. Tseng, S-T. Yau,
{\em Anomaly Cancellation and Smooth Non-Kahler Solutions in
Heterotic String Theory}, Nucl.Phys. {\bf B751} (2006) 108-128.

\bibitem{Berg} E. A. Bergshoeff, M. de Roo, {\em The quartic
effective action of the heterotic string and supersymmetry}, Nucl.
Phys. {\bf B 328} (1989), 439.

\bibitem{Bo} E. Bonan, {\em Sur le vari\'et\'es riemanniennes a groupe
  d'holonomie $G_2$ ou $Spin(7)$}, C. R. Acad. Sci. Paris {\bf 262} (1966),
  127-129.

\bibitem{Br} R. Bryant, {\em Metrics with exeptional holonomy}, Ann.  Math.
  {\bf 126} (1987), 525-576.

\bibitem{Br1}  R. Bryant, {\em Some remarks on $G_2$ structures}, Proceeding of
Gokova Geometry-Topology Conference 2005 edited by S. Akbulut, T
Onder, and R.J. Stern (2006), International Press, 75--109.


\bibitem{BS} R. Bryant, S.Salamon, {\em On the construction of some
complete metrics with exceptional holonomy}, Duke Math. J. {\bf 58}
(1989), 829-850.


\bibitem {C1} F. Cabrera, {\em On Riemannian manifolds with $Spin(7)$-structure},
Publ. Math. Debrecen {\bf 46} (3-4) (1995), 271-283.

\bibitem{Cabr} F. Cabrera, {\em On Riemannian manifolds with $G_2$-structure},
Bolletino UMI A {\bf 10} (7) (1996), 98-112.

\bibitem{Car1}  G. L. Cardoso, G. Curio, G. Dall'Agata, D. Lust, {\em
BPS Action and Superpotential for Heterotic String Compactifications
with Fluxes}, JHEP 0310 (2003) 004.

\bibitem{CCDLMZ} G.L. Cardoso, G. Curio, G. Dall'Agata, D. Lust, P. Manousselis, G. Zoupanos,
{\em Non-K\"aeler string back-grounds and their five torsion
classes}, Nuclear Phys. {\bf B 652} (2003), 5--34.


\bibitem {CS} S. Chiossi, S. Salamon, {\em The intrinsic torsion of SU(3) and
$G_2$-structures}, Differential Geometry, Valencia 2001, World Sci.
Publishing, 2002, pp. 115-133.

\bibitem{ConS}
D. Conti, S. Salamon, {\em Generalized Killing spinors in dimension
$5$}, Trans. Amer. Math. Soc. {\bf 359} (2007), 5319--5343.


\bibitem{CDev} E. Corrigan, C. Devchand, D.B. Fairlie, J. Nuyts, {\em First-order
equations for gauge fields in spaces of dimension greater than
four}, Nuclear Phys. B {\bf 214} (1983),  no. 3, 452--464.

\bibitem{DFG} K. Dasgupta, H. Firouzjahi, R. Gwyn, {\em On the warped heterotic axion},
arXiv:0803.3828 [hep-th], to appear in JHEP.


\bibitem{DRS} K. Dasgupta, G. Rajesh, S. Sethi, {\em M theory,
orientifolds and G-flux}, JHEP {\bf 0211}, 006 (2002).

\bibitem{Bwit} B. de Wit, D.J. Smit, N.D. Hari Dass, {\em Residual Supersimmetry
Of Compactified D=10 Supergravity}, Nucl. Phys. {\bf B 283} (1987),
165.

\bibitem{DT} S.K. Donaldson, R.P. Thomas, {\em Gauge theory in higher dimensions},
 The geometric universe (Oxford, 1996), 31--47, Oxford Univ. Press, Oxford, 1998.


\bibitem {FNu} D.B. Fairlie, J. Nuyts, {\em Spherically symmetric solutions of gauge theories
in eight dimensions}, J. Phys. {\bf A17} (1984) 2867.

\bibitem {F} M. Fern\'andez, {\em A classification of Riemannian manifolds with structure group
$Spin(7)$}, Ann. Mat. Pura Appl. {\bf 143} (1982), 101-122.

\bibitem {FG} M. Fern\'andez, A. Gray, {\em Riemannian manifolds with
  structure group $G_2$}, Ann. Mat. Pura Appl. (4) 32 (1982), 19-45.

\bibitem {FIUV} M. Fern\'andez, S. Ivanov, L. Ugarte, R. Villacampa, {\em Non-Kaehler
heterotic-string compactifications with non-zero fluxes and constant
dilaton}, Commun. Math. Phys. (to appear), arXiv:0804.1648.

\bibitem {FIUV2} M. Fern\'andez, S. Ivanov, L. Ugarte, R. Villacampa, {\em Compact supersymmetric
solutions of the heterotic equations of motion in dimension 5},
arXiv:0811.2137.

\bibitem {FTUV} M. Fern\'andez, A. Tomassini, L. Ugarte, R. Villacampa, {\em Balanced hermitian
metrics from $SU(2)$-structures}, arXiv:0808.1201.


\bibitem {FUg} M. Fern\'andez, L. Ugarte, {\em Dolbeault cohomology for $G\sb 2$-manifolds},
Geom. Dedicata {\bf 70} (1998), no. 1, 57--86.


\bibitem{FGW} D.Z. Freedman, G.W. Gibbons, P.C. West, {\em Ten Into Four Won't Go},
Phys. Lett. {\bf B 124} (1983), 491.

\bibitem{F3} Th. Friedrich {\ On types of non-integrable geometries},   Rend. Circ. Mat. Palermo (2) Suppl.  No.  {\bf 71}  (2003), 99--113.

\bibitem{Fr7} Th. Friedrich, {\em $G_2$-Manifolds With Parallel Characteristic Torsion},
 Diff. Geom. Appl.  {\bf 25}  (2007),  no. 6, 632--648.

\bibitem {FI} Th. Friedrich, S. Ivanov {\em Parallel spinors and connections
  with skew-symmetric torsion in string theory}, Asian J. Math. {\bf 6} (2002), 303-336.


\bibitem{FI1} Th. Friedrich, S.Ivanov, {\em Killing spinor equations in
  dimension 7 and geometry of integrable $G_2$ manifolds}, J. Geom. Phys
  {\bf 48} (2003), 1-11.

\bibitem{FI2} Th. Friedrich, S. Ivanov, {\em Almost contact manifolds, connections with torsion,
parallel spinors}, J. reine angew. Math. {\bf 559} (2003), 217-236.

\bibitem{y2} J-X. Fu, S-T. Yau, {\em Existence of Supersymmetric Hermitian Metrics
with Torsion on Non-Kaehler Manifolds}, arXiv:hep-th/0509028.

\bibitem{y3} J-X. Fu, S-T. Yau, {\em The theory of superstring with flux on non-K\"ahler
manifolds and the complex Monge-Ampere equation},  J. Diff. Geom.  {\ bf 78}  (2008),  no. 3, 369--428.

\bibitem {FN} S. Fubini, H. Nikolai, {\em The octonionic instanton}, Phys. Let. B
{\bf 155} (1985) 369.

\bibitem {GKMW} J. Gauntlett, N. Kim, D. Martelli, D. Waldram, {\em Fivebranes
  wrapped on SLAG three-cycles and related geometry}, JHEP 0111 (2001) 018.

\bibitem {GMPW} J.P. Gauntlett, D. Martelli, S. Pakis, D. Waldram,
{\em  G-Structures and Wrapped NS5-Branes}, Commun. Math. Phys. {\bf
247} (2004), 421-445.

\bibitem {GMW} J. Gauntlett, D. Martelli, D. Waldram, {\em Superstrings with
  Intrinsic torsion}, Phys. Rev. {\bf D69} (2004) 086002.

\bibitem{Gibb} G.W. Gibbons, D.N. Page, C.N. Pope, {\em Einstein metrics
on $S^3,\mathbb R^3,$ and $\mathbb R^4$ bundles}, Commun. Math.
Phys. {\bf 127} (1990), 529-553.

\bibitem{GPap} J. Gillard, G. Papadopoulos, D. Tsimpis, {\em
Anomaly, Fluxes and (2,0) Heterotic-String Compactifications}, JHEP
0306 (2003) 035.

\bibitem{GP} E. Goldstein, S. Prokushkin, {\em Geometric Model for Complex
Non-K\"aehler Manifolds with SU(3) Structure}, Commun. Math. Phys.
{\bf 251} (2004), 65--78.

\bibitem {Gr} A. Gray, {\em Vector cross product on manifolds}, Trans. Am.
Math. Soc.  {\bf 141} (1969), 463-504, Correction {\bf 148} (1970),
625.

\bibitem {GLP} U. Gran, P. Lohrmann, G. Papadopoulos, {\em The spinorial geometry
of supersymmetric heterotic string backgrounds}, JHEP 0602 (2006)
063.

\bibitem {GPR} U. Gran, G. Papadopoulos, D. Roest {\em Supersymmetric
heterotic string backgrounds}, Phys. Lett. B {\bf 656} (2007), 119.

\bibitem {GPRS} U. Gran, G. Papadopoulos, D. Roest, P. Sloane {\em Geometry
of all supersymmetric type I backgrounds}, JHEP 0708 (2007) 074.

\bibitem {GNic} M. G\"unaydin, H. Nikolai, {\em Seven-dimensional octonionic Yang-Mills
instanton and its extension to an heterotic string soliton}, Phys.
Lett. B {\bf 353} (1991) 169.

\bibitem {GIP} J. Gutowski, S. Ivanov, G. Papadopoulos,
{\em Deformations of generalized calibrations and compact non-Kahler
manifolds with vanishing first Chern class}, Asian J. Math. {\bf 7}
(2003), 39-80.

\bibitem {HL} R. Harvey, H.B. Lawson, {\em Calibrated geometries}, Acta Math. {\bf 148}
(1982), 47-157.

\bibitem {HS} J.A. Harvey, A. Strominger, {\em Octonionic superstring solitons},
Phys. Review Let. {\bf 66} 5 (1991) 549.

\bibitem{HP1} P.S. Howe, G. Papadopoulos, {\em Ultraviolet
behavior of two-dimensional supersymmetric non-linear sigma models},
Nucl. Phys. {\bf B 289} (1987), 264.

\bibitem{Hull} C.M. Hull, {\em Anomalies, ambiquities and
superstrings}, Phys. lett. {\bf B167} (1986), 51.

\bibitem{HT} C.M. Hull, P.K. Townsend, {\em The two loop beta
function for sigma models with torsion}, Phys. Lett. {\bf B 191}
(1987), 115.

\bibitem{HuW} C.M. Hull, E. Witten, {\em Supersymmetric sigma
models and the Heterotic String}, Phys. Lett. {\bf B 160} (1985),
398.

\bibitem{II} P. Ivanov, S. Ivanov, {\em SU(3)-instantons and
$G_2,Spin(7)$-Heterotic string solitons}, Comm. Math. Phys. {\bf
259} (2005), 79-102.

\bibitem{I2} S. Ivanov, {\em Geometry of quaternionic K\"{a}hler connections with torsion},
J. Geom. Phys. {\bf 41} (2002), no. 3, 235--257.

\bibitem {I1} S. Ivanov, {\em Connection with torsion, parallel spinors and geometry of Spin(7) manifolds},
Math. Res. Lett.  {\bf 11} (2004), no. 2-3, 171--186.

\bibitem {IP2} S. Ivanov, G. Papadopoulos, {\em A no-go theorem for string warped compactifications},
 Phys.Lett. {\bf B497} (2001) 309-316.

\bibitem {IP1} S. Ivanov, G. Papadopoulos, {\em Vanishing Theorems and String Backgrounds},
 Class.Quant.Grav. {\bf 18} (2001) 1089-1110.

\bibitem{J1} D. Joyce, {\em Compact Riemannian 7-manifolds with holonomy
  $G_2$. I}, J.Diff. Geom.  43 (1996), 291-328.

\bibitem{J2} \bysame, {\em Compact Riemannian 7-manifolds with holonomy
  $G_2$. II}, J.Diff. Geom.,  43 (1996), 329-375.

\bibitem{J3} \bysame, Compact Riemannian manifolds with special holonomy,
  Oxford University Press, 2000.

\bibitem{Kob} S.Kobayashi, {\em Principal fibre bundles with
1-dimensional toroidal group}, Tohoku Math. J. (56)  {\bf 8} (1956),
29-45.

\bibitem {Kov} A. Kovalev, {\em Twisted connected sums and special
  Riemannian holonomy}, J. Reine Angew. math. {\bf 565} (2003), 125-160.

\bibitem {LM} B. Lawson, M.-L.Michelsohn, Spin Geometry, Princeton
  University Press, 1989.

\bibitem{y1} J. Li, S-T. Yau, {\em The Existence of Supersymmetric String Theory with Torsion},
 J. Diff. Geom. {\bf 70}, no. 1, (2005).

\bibitem{Nag}  P.-A. Nagy     {\em Prolongations of Lie algebras and applications},
    arXiv:0712.1398.


\bibitem {P} G. Papadopoulos, {\em Solution of heterotic Killing spinor equations
and special geometry}, arXiv:0811.1539 [math.DG].


\bibitem{RC} R.  Reyes Carri\'{o}n, {\em A generalization of the notion of instanton},
Diff. Geom. Appl. {\bf 8} (1998), no. 1, 1--20.

\bibitem{Sen} A. Sen, {\em (2,0) supersymmetry and space-time
supersymmetry in the heterotic strin theory}, Nucl. Phys. {\bf B
167} (1986), 289.

\bibitem{Sal} S. Salamon, Riemannian geometry and holonomy groups, Pitman
  Res.  Notes Math. Ser., 201 (1989).

\bibitem{Scho} N. Schoemann, {\ Almost Hermitian Structures with Parallel Torsion},
  J. Geom. Phys.  {\ bf 57}  (2007),  no. 11, 2187--2212.

\bibitem {Str} A. Strominger, {\em Superstrings with torsion}, Nucl. Phys.
{\bf B 274} (1986) 253.

\bibitem {Ug} L. Ugarte, {\em Coeffective Numbers of Riemannian 8-manifold with
Holonomy in} $Spin(7)$, Ann. Glob. Anal. Geom. {\bf 19} (2001),
35-53.



\end{thebibliography}
\end{document}